\documentclass[12pt]{amsart}
\usepackage{diagrams}
\newtheorem{theorem}{Theorem}[section]
\newtheorem{prop}[theorem]{Proposition}
\newtheorem{lemma}[theorem]{Lemma}

\newtheorem{cor}[theorem]{Corollary}
\newtheorem{remark}[theorem]{Remark}

\newcommand{\s}{\mathfrak{s}}
\newcommand{\rr}{\mathfrak{r}}
\newcommand{\kk}{\mathfrak{k}}
\renewcommand{\tt}{\mathfrak{t}}
\newcommand{\im}{\mbox{im}}

\newcommand{\spinc}{{\mbox{spin$^c$} }}
\newcommand{\OS}{{Ozsv\'ath-Szab\'o} }
\newcommand{\OnS}{{Ozsv\'ath and Szab\'o}}
\newcommand{\zee}{\mathbb{Z}}
\newcommand{\arr}{\mathbb{R}}

\newcommand{\cee}{\mathbb{C}}
\newcommand{\eff}{\mathbb{F}}
\newcommand{\R}{{\mathcal R}}
\newcommand{\K}{{\mathcal K}}
\newcommand{\T}{{\mathcal T}}

\renewcommand{\O}{{\mbox{\normalfont\gothfamily O}}}
\newcommand{\cO}{\mathcal O}
\renewcommand{\L}{{\mathcal L}}
\newcommand{\M}{{\mathcal M}}
\newcommand{\ts}{\textstyle}
\newcommand{\ul}{\underline}
\newcommand{\tX}{\widetilde{X}}
\newcommand{\tK}{\widetilde{K}}
\newcommand{\tSigma}{\widetilde{\Sigma}}
\newcommand{\tZ}{\widetilde{Z}}

\newcommand{\tkk}{\tilde{\kk}}
\newcommand{\cpbar}{\overline{\cee P}^2}

\usepackage{graphicx,amsfonts,yfonts}

\begin{document}

\title{Knotted surfaces in 4-manifolds}
\author{Thomas E. Mark}

\begin{abstract} Fintushel and Stern have proved that if $\Sigma\subset X$ is a symplectic surface in a symplectic 4-manifold such that $\Sigma$ has simply-connected complement and nonnegative self-intersection, then there are infinitely many topologically equivalent but smoothly distinct embedded surfaces homologous to $\Sigma$. Here we extend this result to include symplectic surfaces whose self-intersection is bounded below by $2-2g$, where $g$ is the genus of $\Sigma$. 

We make use of tools from Heegaard Floer theory, and include several results that may be of independent interest. Specifically we give an analogue for \OS invariants of the Fintushel-Stern knot surgery formula for Seiberg-Witten invariants, both for closed 4-manifolds and manifolds with boundary. This is based on a formula for the \OS invariants of the result of a logarithmic transformation, analogous to one obtained by Morgan-Mrowka-Szab\'o for Seiberg-Witten invariants, and the results on \OS invariants of fiber sums due to the author and Jabuka. In addition, we give a calculation of the twisted Heegaard Floer homology of circle bundles of ``large'' degree over Riemann surfaces.\end{abstract}

\maketitle

\section{Introduction}

Some time ago, Fintushel and Stern introduced a technique they called ``rim surgery'' for changing the embedding of a smooth surface $\Sigma$ in a closed 4-manifold $X$ \cite{FSrim, FSaddend}. Their construction makes use of a knot $K\subset S^3$, and can be seen as an instance of their earlier ``knot surgery'' construction, applied to the complement of $\Sigma$. The interesting aspect of the construction is that under suitable conditions, Fintushel and Stern were able to show that the resulting surface $\Sigma_K\subset X$ is topologically equivalent to $\Sigma$, but smoothly distinct: rim surgery results in a {\it smooth} knotting of $\Sigma$ but not a {\it topological} one. 

To ensure that the topological type of $(X,\Sigma)$ is unchanged by the construction, it suffices to assume that the complement $Z$ of a regular neighborhood of $\Sigma$ is simply connected (see \cite{FSrim,boyer,freedman}, also Remark \ref{pi1remark} below). Our current concern is the problem of distinguishing the smooth types of $(X,\Sigma)$ and $(X,\Sigma_K)$. Using Seiberg-Witten theory, Fintushel and Stern were able to show that if $X$ is a symplectic 4-manifold and $\Sigma\subset X$ is a symplectic surface, and if the self-intersection $\Sigma.\Sigma$ is nonnegative, then $(X,\Sigma)$ and $(X,\Sigma_K)$ are smoothly distinct whenever $K$ has nontrivial Alexander polynomial.

Our purpose here is to revisit the rim surgery construction using the tools of Heegaard Floer theory and \OS 4-manifold invariants.  In particular, we have the following result.

\begin{theorem}\label{mainthm} Let $X$ be a closed symplectic 4-manifold and $\Sigma_0 \subset X$ a smoothly embedded symplectic surface of genus $g\geq 1$. Assume that
\begin{itemize}
\item $\pi_1(X\setminus \Sigma_0) = 1$, and 
\item the self-intersection of $\Sigma_0$ satisfies $\Sigma_0.\Sigma_0 \geq 2-2g$. 
\end{itemize}
Then there exist infinitely many smooth surfaces $\Sigma_n\subset X$, $n=1,2,\ldots$, homologous to $\Sigma_0$, that are topologically equivalent to $\Sigma_0$ but smoothly inequivalent. That is, there exist homeomorphisms of pairs $(X,\Sigma_n)\simeq (X,\Sigma_m)$ for all $n, m\geq 0$, but no diffeomorphisms between these pairs unless $n = m$.
\end{theorem}

This theorem strengthens Fintushel and Stern's result to the extent that the square $\Sigma_0.\Sigma_0$ is allowed to be negative. 

It is not hard to construct symplectic surfaces $\Sigma$ in minimal symplectic manifolds, having simply-connected complement, but violating the condition on $\Sigma.\Sigma$ given in the theorem. It is an intriguing question whether rim surgery changes the smooth embedding of such surfaces. 

The knotted surface $\Sigma_K$ will generally not be symplectic. Indeed, if the Alexander polynomial $\Delta_K(t)$ is not trivial then one can use an argument entirely analogous to one given by Fintushel and Stern in \cite{FSrim} to show that $\Sigma_K$ is not isotopic to a symplectic surface.

The proof of Theorem \ref{mainthm} is based on a number of auxiliary results, which we hope will also be of interest. To give context to these results, we first recall the construction of $\Sigma_K$. Let $c$ be a simple closed curve on $\Sigma$, and $K\subset S^3$ a knot. A neighborhood of $c$ on $\Sigma$ is an annulus $A = S^1\times [-1,1]$, and the normal bundle of $\Sigma$ can be trivialized over this annulus: thus the restriction of the normal bundle over $A$ can be identified with $A\times D^2$. The surface $\Sigma$ cuts each normal tube $\theta\times[-1,1]\times D^2$ (where $\theta \in S^1 \cong c$) in the core arc $L =\theta\times[-1,1]\times 0$. In each tube, replace $L$ by a knotted arc $K_0\subset [-1,1]\times D^2$, whose obvious ``closure'' in $S^3$ is $K$. The union of the resulting copies of $K_0$ gives a knotted annulus in $S^1\times [-1,1]\times D^2$, which can be glued to the complement of $A$ in $\Sigma$ to give a new embedded surface $\Sigma_K$. 

The construction can be rephrased as follows: rather than replacing the core arc $L$ by a knotted arc $K_0$, we remove a neighborhood $V = S^1\times D^2$ of a small linking circle of $L$ and replace it by the exterior $E(K)$ of the knot $K$ (in such a way as to preserve the homology). Performing this operation in each normal tube $\theta\times[-1,1]\times D^2$, we have replaced $S^1\times V = S^1\times S^1\times D^2$ by $S^1\times E(K)$, in such a way that the boundary of the Seifert surface for $K$ is matched with the boundary of a normal disk $pt\times pt \times \partial D^2\subset S^1\times S^1\times D^2$. This construction is an instance of {\it knot surgery} using the torus $T$ given by the product of $c\subset \Sigma$ with the boundary of a (slightly smaller) normal disk.

In a remarkable paper \cite{FSknot}, Fintushel and Stern determined the behavior of the Seiberg-Witten invariants of a closed 4-manifold under knot surgery along a torus satisfying certain conditions. Roughly, they showed that the Seiberg-Witten invariant is multiplied by the Alexander polynomial of $K$. Our first result here is an analogous statement for the \OS 4-manifold invariants $\O_{X}$ (Theorem \ref{closedknotsurg}). The proof follows much the same lines as Fintushel and Stern's original argument based on the skein relation for the Alexander polynomial, together with certain gluing results. One of the latter is the formula for the behavior of \OS invariants under fiber sum obtained in \cite{products}, and the other is a theorem giving the behavior of \OS invariants under logarithmic transformations (Theorem \ref{logtransfthm} and Corollary \ref{logtransfcor} below). 

In fact, using the functoriality of \OS invariants under cobordism we are able to give a version of the knot surgery formula for {\it relative} \OS invariants of 4-manifolds with boundary (Theorem \ref{relknotsurgthm} below, stated more fully as Theorem \ref{knotsurgthm} subsequently). To state it, recall that the relative \OS invariant of a 4-manifold $Z$ with boundary $Y$, in a \spinc structure $\s$, is an element of the twisted Floer homology of $Y$:
\[
\Psi_{Z,\s}\in HF^-(Y,\s; M(Z)).
\]
Here $M(Z)$ is the coefficient module $\zee[H^1(Y)/H^1(Z)]$ induced by $Z$. This, and the Floer homology itself, is a module for the group ring $R_Y = \zee[H^1(Y)]$; the relative invariant is well-defined up to an overall sign and up to multiplication by elements of $H^1(Y)$. For the following result, we use brackets $[\Psi_{Z,\s}]$ to indicate the image of $\Psi_{Z,\s}$ in twisted $\zee/2\zee$-coefficient Floer homology $HF^-(Y,\s; M\otimes_\zee \zee/2\zee)$.

\begin{theorem}\label{relknotsurgthm} Suppose $(Z,\s)$ is an oriented \spinc 4-manifold with boundary $Y$, and let $T\subset Z$ be an embedded torus such that
\begin{itemize}
\item $T$ represents an element of infinite order in $H_2(Z;\zee)$, and
\item The homology class $[T]$ lies in the image of the inclusion $H_2(Y;\zee)\to H_2(Z;\zee)$.
\end{itemize}
Let $Z_K$ denote the result of knot surgery on $Z$ using the knot $K\subset S^3$ and the torus $T$. Then the relative \OS invariant of $Z_K$ is related to that of $Z$ by
\begin{equation}\label{surgformula}
[\Psi_{Z_K,\s}] = [\Delta_K(t)\cdot \Psi_{Z,\s}]
\end{equation}
where $\Delta_K$ is the Alexander polynomial of $K$, and $t\in H^1(Y;\zee)$ is Poincar\'e dual to any element of $H_2(Y;\zee)$ mapping to $[T]$.
\end{theorem}

Note that since the relative invariant $\Psi_{Z,\s}$ is defined only up to the action of $H^1(Y)$, that is, up to multiplication by a unit in $R_Y$, the above theorem determines the relative invariant of $Z_K$ modulo this indeterminacy. 

The final ingredient in the proof of Theorem \ref{mainthm} is the calculation of the Heegaard Floer homology of any circle bundle over a surface, whose degree is sufficiently large in absolute value. This calculation was carried out for Floer homology with integer coefficients by {\OnS} in \cite{OSknot}; our result (Theorem \ref{calcthm} below) deals with Floer homology with ``univerally twisted'' coefficients. Since these circle bundles arise in a number of natural contexts we expect that our results, phrased in terms of certain syzygy modules associated to a free resolution over $R_Y$ of $\zee$, will find further application.

\subsection*{Outline of proof of Theorem \ref{mainthm}} We can choose a sequence of knots $K_n$ whose Alexander polynomials are all distinct even after reducing coefficients modulo 2; then we take $\Sigma_n$ to be the result $\Sigma_{K_n}$ of rim surgery on $\Sigma_0$ using $K_n$. Fixing some choice of knot $K$, our main goal is to distinguish the diffeomorphism types of the complements $Z$ and $Z_K$ of $\Sigma$ ($=\Sigma_0$) and $\Sigma_K$. As indicated previously, the manifolds $Z$ and $Z_K$ are related by knot surgery along a boundary-parallel torus. With the description given in Theorem \ref{relknotsurgthm} of the relation between the relative invariants of these two manifolds, the key issues in this task are as follows:
\begin{enumerate}
\item[i)] Ascertaining that the relative invariant $\Psi_Z$ is nonzero, and
\item[ii)] Verifying that multiplication by the Alexander polynomial as in \eqref{surgformula} has a nontrivial effect.
\end{enumerate}
To arrange for (i) to hold, we assume that $X$ is a symplectic manifold and therefore has a nonvanishing \OS invariant \cite{OSsymp}. Then, if $\Sigma.\Sigma\geq 0$ (and assuming for simplicity that $b^+(X)\geq 2$), the \OS invariant of $X$ can be calculated by pairing the relative invariants of the two pieces of the decomposition $X = Z \cup nbd(\Sigma)$, and (i) follows immediately. For the case of negative square, assume that $\Sigma$ is a symplectic surface in $X$: then if we can find a standard manifold $E$ having $b^+(E)\geq 2$ and containing a symplectic surface $\Sigma'$ with $\Sigma'.\Sigma' = -\Sigma.\Sigma$, we can form a symplectic fiber sum $M = X\#_{\Sigma = \Sigma'} E$. The \OS invariant of $M$ is nontrivial since $M$ is symplectic, and it can be calculated by pairing relative invariants from $X \setminus nbd(\Sigma) = Z$ and $E \setminus nbd(\Sigma')$, since both sides have nonvanishing $b^+$. We again conclude that $\Psi_Z \neq 0$. However, the adjunction inequality implies that the self-intersection of the surface $\Sigma'$ in the manifold $E$ is bounded above by $2g-2$, hence this argument works only if $\Sigma.\Sigma\geq 2-2g$.

To prove (ii) above, we must calculate the twisted Heegaard Floer homology of the boundary $Y = \partial Z = \partial Z_K$ as a module over $R_Y = \zee[H^1(Y)]$. This boundary is a circle bundle over $\Sigma$, having degree $-n$ if $\Sigma.\Sigma = n$ (since $Y$ is oriented as the boundary of the complement). The Floer homology of circle bundles of large degree is reasonably straightforward to calculate using techniques due to \OnS; in particular if $n=\Sigma.\Sigma < 2-2g$ then we prove below that the ``relevant part'' of the Floer homology of the circle bundle of degree $-n$ is a free module of rank 1 over $R_Y$. Equation \eqref{surgformula} then shows that rim surgery is easily detected in this case: the relative invariant changes by multiplication by a (non-unit) element of $R_Y$, which certainly never has a trivial effect even up to automorphism. 

Of course, the necessary assumptions on the self-intersection of $\Sigma$ in (i) and (ii) are precisely opposite. To bridge the gap, we arrange first that $\Sigma.\Sigma = 2-2g$ by blowing up to reduce the self-intersection, so that (i) still applies to give a nonzero relative invariant. Then we argue that one further blowup, to put us in the situation amenable to point (ii), preserves the nonvanishing relative invariant. This is not quite as straightforward as an application of the ``blowup formula'' for \OS invariants, as the complement of the blown-up surface $\tSigma$ is not a blowup of the complement of $\Sigma$. In particular, it relies on the structure of the twisted Floer homology of circle bundles.

\begin{remark}\label{pi1remark} H. J. Kim and D. Ruberman \cite{kimrub,kimrub2} have considered the effect of rim surgery on the topological type of the embedding of a surface $\Sigma\subset X$, in the case that the complement has nontrivial fundamental group. For example, they show that if $\Sigma$ is a homologically essential surface in $X$ such that $X$ is simply connected and $\pi_1(Z)$ and $\pi_1(Z_K)$ are cyclic (where $Z = X\setminus nbd(\Sigma)$ and $Z_K = X\setminus nbd(\Sigma_K)$ as above), then the result $\Sigma_K$ of rim surgery is topologically isotopic to $\Sigma$. The results of this paper apply equally to this situation to distinguish the smooth types of $(X,\Sigma)$ and $(X,\Sigma_K)$ when $\pi_1(X\setminus\Sigma)$ is nontrivial, assuming that $\Sigma.\Sigma\geq 2-2g$ (and the rest of the setup of Theorem \ref{mainthm}). In fact formula \eqref{surgformula} requires no assumption on fundamental group, though there are technical considerations in the blowup argument mentioned in the preceding paragraph that oblige us to assume that the restriction $H^1(Z)\to H^1(\partial Z)$ is trivial. Since $H^1(Z) = H^1(X)$ when $\Sigma$ is homologically essential, this is certainly true when $X$ is simply connected.
\end{remark}

\subsection*{Organization} The remainder of the paper proceeds as follows. Section \ref{backgroundsec} begins with relevant background material on Heegaard Floer theory, particularly the cut-and-paste results obtained by the author and S. Jabuka in \cite{products}. Then we give a result describing the effect of a logarithmic transformation on both the relative and absolute \OS 4-manifold invariants, analogous (in the closed case) to a formally similar result in Seiberg-Witten theory due to Morgan, Mrowka, and Szab\'o \cite{MMS}. Here we encounter the sign issue mentioned above: at the moment, the sign indeterminacy of the \OS invariants means that our formula is proved only modulo 2.

Section \ref{knotsurgsec} deduces the knot surgery formula in the cases of closed manifolds and manifolds with boundary, based on the results on logarithmic transformation and those on fiber sums from \cite{products}. The sign issue propagates here to give no better than a mod-2 statement.

Section \ref{proofsec} details points (i) and (ii) above: we prove that the complements of certain symplectic surfaces have nonvanishing relative invariants in section \ref{relinvtsec}, and we calculate the twisted Floer homology of circle bundles of large degree in section \ref{calcsec}. 

Finally in section \ref{blowupsec} we ``bridge the gap'' by showing that blowups preserve certain relative invariants, and we spell out the proof of Theorem \ref{mainthm}.

\section{Background and Preliminary Results}\label{backgroundsec}

Here we recall the basic ideas of Heegaard Floer theory for 3-manifolds and the associated invariants of 4-manifolds, set notation, and prove some basic results needed below. For a more detailed introduction to Heegaard Floer homology, the reader is referred to the growing literature on the subject.

\subsection{Invariants for 3- and 4-manifolds}
If $Y$ is a closed, oriented 3-manifold with \spinc structure $\s$, we have a collection of Heegaard Floer homology groups $HF^+(Y,\s)$, $HF^-(Y,\s)$, $HF^\infty(Y,\s)$, $\widehat{HF}(Y,\s)$, which are (relatively) graded modules over $\zee$. More generally, let $R_Y = \zee[H^1(Y;\zee)]$, and suppose $M$ is a module for $R_Y$. Then there is a notion of Heegaard Floer homology with coefficients in $M$, e.g. $HF^-(Y,\s;M)$, and in particular we can consider the ``fully twisted'' Floer homology group $HF^-(Y,\s; R_Y)$ (sometimes also written $\underline{HF}^-(Y,\s)$). This theory was introduced and developed by Ozsv\'ath and Szab\'o; see \cite{OS3} in particular for much of the material below. 

We will also need to refer to the ``perturbed'' version of Heegaard Floer homology as described in \cite{products}, which makes use of a class $\eta\in H^2(Y;\arr)$. Given such a class, one can form the {\it Novikov ring} $\R_{Y,\eta}$ as a completion of the group ring $R_Y$, namely
\[
\R_{Y,\eta} = \{ \sum_{g\in H^1(Y;\zee)} a_g\cdot g\, |\, a_g\in \zee\}\subset \zee[[H^1(Y;\zee)]],
\]
where for each $N\in \zee$ there are only finitely many $g\in H^1(Y;\zee)$ such that $a_g$ is nonzero and $\langle\eta\cup g, [Y]\rangle \leq N$ (we say that $\R_{Y,\eta}$ is the {\it Novikov completion} of $R_Y$ with respect to the function $\eta\cup \cdot: H^1(Y)\to \arr$). As before, if $\M$ is a module for $\R_{Y,\eta}$ we have perturbed Floer homology $HF^-(Y,\s;\M)$ with coefficients in $\M$. A key property of $\R_{Y,\eta}$ is that it is flat as a module over $R_Y$; this has the consequence that if $M$ is a module for $R_Y$ then the associated module $\M = M\otimes_{R_Y} \R_{Y,\eta}$ has Floer homology 
\[
HF^-(Y,\s;\M) = HF^-(Y,\s, M)\otimes_{R_Y} \R_{Y,\eta}.
\]

Heegaard Floer homology is functorial, in the sense that if $W: Y_1\to Y_2$ is a 4-dimensional cobordism with \spinc structure $\rr$, then there is an induced homomorphism $F_{W,\rr}: HF^-(Y_1,\s_1;\zee)\to HF^-(Y_2,\s_2;\zee)$, where $\s_i = \rr|_{Y_i}$. More generally if $M$ is a module for $R_{Y_1}$, there is an induced module $M(W)$ for $R_{Y_2}$ given by
\[
M(W) = M\otimes_{R_{Y_1}} \zee[K(W)],
\]
where $K(W) = \ker(H^2(W,\partial W)\to H^2(W)) \cong H^1(\partial W)/ H^1(W)$, and the action of $R_{Y_1}$ and $R_{Y_2}$ on $\zee[K(W)]$ is by the coboundary maps $H^1(Y_i)\to H^2(W,\partial W)$. The cobordism then induces homomorphisms
\[
F_{W,\rr}: HF(Y_1,\s_1; M)\to HF(Y_2,\s_2; M(W)),
\]
which are well-defined up to sign and multiplication by units in $R_{Y_1}$ and $R_{Y_2}$, i.e., the action by classes in $H^1(Y_1)\oplus H^1(Y_2)$.
If we are given a class $\eta\in H^2(W;\arr)$, we can carry out the construction in the perturbed setting to obtain a homomorphism $F_{W,\s,\eta}^- : HF^-(Y_1,\s_1; \M)\to HF^-(Y_2,\s_2; \M(W,\eta))$, where $\M$ is a module for $\R_{Y_1,\eta_1}$, $\eta_1 = \eta|_{Y_1}$, and
\[
\M(W,\eta) = \M \otimes _{\R_{Y_1,\eta_1}} \K(W,\eta).
\]
Here $\K(W,\eta)$ is the Novikov completion of $\zee[K(W)]$ with respect to the linear function on $K(W)\subset H^2(W,\partial W)$ induced by cup product with $\eta$.

Now suppose $Z$ is an oriented 4-manifold with connected boundary $Y$ and $\s$ is a \spinc structure on $Z$. We can associate to $Z$ several versions of a {\it relative \OS invariant}: the ``twisted'' invariant $\Psi_{Z,\s}\in HF^-(Y,\s|_Y; \zee[K(Z)])$, an ``untwisted'' version $\Psi^{un}_{Z,\s} \in HF^-(Y,\s|_Y;\zee)$, and if $\eta\in H^2(Z;\arr)$ a ``perturbed'' version $\Psi_{Z,\s,\eta} \in HF^-(Y,\s|_Y; \K(Z,\eta))$. Explicitly, let $\Theta^-$ be a generator in top degree of $HF^-(S^3)$ (in the unique \spinc structure on $S^3$; note that since $H^1(S^3) = H^2(S^3) = 0$, no nontrivial twisting or perturbation is possible). Thus $\Theta^-$ is well-defined up to sign. Then we define
\[
\Psi_{Z,\s} = F^-_{Z,\s}(\Theta^-) \in HF^-(Y,\s|_Y; \zee[K(Z)]),
\]
where here and below we adopt the convention that if $Z$ is a 4-manifold with a single boundary component $Y$, then $F_{Z,s}$ refers to the map in Floer homology induced by removing a 4-ball from the interior of $Z$ and regarding the result as a cobordism $S^3\to Y$. The other versions of the relative invariant are defined analogously, using the induced homomorphism in untwisted or perturbed Floer homology as appropriate. In each case, the relative invariant is well-defined up to a sign and translation by the action of an element of $H^1(Y)$ (this action is trivial in the untwisted case).

\OS invariants for closed 4-manifolds are constructed as follows. First, recall \cite[section 4]{products} that there is a bilinear pairing
\begin{equation}\label{pairingdef}
\langle \cdot, \cdot\rangle: HF^+(Y, \s; M)\otimes_{R_Y} HF^-(-Y,\s; N)\to M\otimes_{R_Y} N
\end{equation}
for any \spinc 3-manifold $(Y,\s)$ and $R_Y$-modules $M, N$. Strictly speaking the pairing is ``antilinear'' in the second variable, meaning that $\langle x, g\cdot y\rangle = \bar{g}\cdot\langle x,y\rangle$, where $\bar{g}$ denotes the image of $g\in R_Y$ under the natural involution induced by negation in $H^1(Y)$.

Next, it was shown in \cite{OS3} that if $b^+(W)>0$ then $F^-_{W}$ takes values in the reduced Floer homology $HF^-_{red}(Y_2)$, while $F^+_W$ factors through $HF^+_{red}(Y_1)$. Moreover, the pairing above descends to a pairing on the reduced modules. Thus if $X$ is a closed 4-manifold with $b^+(X)\geq 2$ we can split $X$ along a 3-manifold $Y$ such that $X = Z_1\cup_Y Z_2$ with $b^+(Z_i)\geq 1$ and define
\begin{equation}\label{Odef}
\O_{X,\s} = \langle \tau^{-1}(\Psi_{Z_1,s_1}), \Psi_{Z_2,s_1}\rangle \in \zee[K(Z_1)]\otimes_{R_Y}\zee[K(Z_2)],
\end{equation}
where $\tau: HF^+_{red}(Y)\to HF^-_{red}(Y)$ is the natural isomorphism. According to \cite[section 2.3]{products} there is an identification of $R_Y$-modules
\[
\zee[K(Z_1)]\otimes_{R_Y}\zee[K(Z_2)] \cong \zee[K(X,Y)],
\]
where $K(X,Y) = \ker(H^2(X)\to H^2(Z_1)\oplus H^2(Z_2))$ is the kernel of the Mayer-Vietoris homomorphism. Thus for a given \spinc structure $\s$ on $X$, the invariant $\O_{X,\s}$ is a Laurent polynomial whose variables can be identified with the set of \spinc structures on $X$ whose restrictions to $Z_1$ and $Z_2$ agree with those of $\s$, and whose coefficients can be thought of as numerical invariants of those \spinc structures. Indeed, if $Y$ is an ``admissible cut'' for $X$ in the sense that $K(X,Y)$ is trivial, then $\s$ is determined uniquely by its restrictions to $Z_1$ and $Z_2$ and $\O_{X,\s}$ is a single integer called the {\it \OS invariant} of $(X,\s)$. The latter is the original definition of the \OS invariant in \cite{OS3}; for the interpretation above using twisted coefficients see \cite{products} (Theorem 7.6). The integer-valued \OS invariant associated to $\s\in Spin^c(X)$ is typically denoted $\Phi_{X,\s}$.

The assumption that $b^+(Z_1)$ and $b^+(Z_2)$ are both nonzero can be relaxed if one uses the perturbed theory, i.e., if we assume that there exists a class $\eta\in H^2(X;\arr)$ that restricts nontrivially to $Y$, and use a pairing on the perturbed Floer homologies analogous to the one described above. This construction allows extension of the definition of \OS invariants to 4-manifolds with $b^+(X) = 1$, though in this case the quantity $\O_{X,\s,\eta}$ depends in general on the data $(Y,\eta)$ and also lies in a Novikov ring rather than a group ring (again, see \cite{products} (particularly section 8) for details).

Finally, we remark that the unperturbed theory can be recovered from the perturbed one by taking $\eta = 0$ throughout.

\subsection{Logarithmic transformations}
Suppose $Z$ is a 4-manifold, possibly with boundary, and $T\subset Z$ is a smoothly embedded torus with trivial normal bundle. A (generalized) {\it logarithmic transformation} along $T$ is obtained by choosing an identification of a tubular neighborhood of $T$ with $T\times D^2$, and forming $Z_\phi = (Z\setminus T\times D^2)\cup_\phi (T\times D^2)$, where $\phi$ is any diffeomorphism of the boundary 3-torus. When $T$ represents a nontrivial class in both $H_2(Z;\arr)$ and $H_2(Z_\phi,\arr)$, the gluing formulas of \cite{products} can be used to understand the behavior of the \OS invariants of $Z$ under this operation. However, we will be interested here in the case of nullhomologous tori in closed manifolds, and since it represents little extra effort, the effect of a logarithmic transformation on the twisted relative invariants of 4-manifolds with boundary.

A formula of the type we have in mind is difficult to state in full generality, so we introduce the following device. For a torus $T\subset Z$ in a 4-manifold with boundary $Y$ as above, we define the {\it $T$-averaged relative invariant} of $Z$ in a given \spinc structure $\s$ as follows. Let $W$ be the complement in $Z$ of a regular neighborhood of $T$, thought of as a cobordism $T^3\to Y$ from the three-torus $T^3 = \partial(nbd(T))$ to $Y = \partial Z$. Observe that there exists a factorization $H^1(Z)\to H^1(W)\to H^1(Y)$, hence a natural quotient $H^1(Y)/H^1(Z) \to H^1(Y)/H^1(W)$. In the notation of the previous subsection $K(Z) = H^1(Y)/H^1(Z)$; we write $K_+(W)$ for the group $H^1(Y)/H^1(W)$. Let 
\[
p: \zee[K(Z)] = \zee\left[{\ts\frac{H^1(Y)}{H^1(Z)}}\right]\to \zee\left[{\ts \frac{H^1(Y)}{H^1(W)}}\right] = \zee[K_+(W)]
\]
be the homomorphism induced by the quotient, and write $p_*$ for the induced homomorphism in twisted Heegaard Floer homology.

The $T$-averaged relative invariant of $Z$, written $\Psi_{Z/T,\s}$, is defined to be the image of $\Theta^-$ under the composition of the map $F^-_{T^2\times D^2, \s}: HF^-(S^3)\to HF^-(T^3; \zee)$ in untwisted Floer homology with the twisted-coefficient map $\ul{F}^-_{W,\s}: HF^-(T^3; \zee)\to HF^-(Y; \zee[K_+(W)])$. Thus
\begin{equation}\label{avgdef}
\Psi_{Z/T, \s} =  \ul{F}^-_{W,\s}( F^-_{T^2\times D^2,\s}(\Theta^-)) \in HF^-(Y,\s; \zee[K_+(W)]),
\end{equation}
where in each instance the symbol ``$\s$'' refers to the \spinc structure $\s$ on $Z$ restricted to the appropriate submanifold (and the relative invariant is defined, as usual, up to sign and the action of $H^1(Y)$). Our terminology for $\Psi_{Z/T,\s}$ is justified by the following, whose proof is an exercise in the formal properties of Heegaard Floer theory \cite{OS3}.

\begin{lemma}\label{Tavglemma} The $T$-averaged relative invariant of $Z$ is related to the usual relative invariant by
\[
\Psi_{Z/T,\s} =p_* \sum_{\s'}\Psi_{Z,\s'},
\]
where the sum is over all \spinc structures $\s'\in Spin^c(Z)$ whose restrictions to $W$ and to $T\times D^2$ agree with those of $\s$. Equivalently, the sum is over all \spinc structures of the form $\s' = \s + \delta h$ for $h\in H^1(\partial(T\times D^2))$ and $\delta$ the Mayer-Vietoris coboundary.
\end{lemma}

More precisely, we should say that there are choices of representatives for the relative invariants $\Psi_{Z,\s'}$ such that the stated formula holds up to multiplication by a unit in $R_Y$.

A technical advantage to working with the $T$-averaged invariant is that the averaged invariant of the result $Z_\phi$ of a logarithmic transformation along $T$ takes values in the same Floer group $HF^-(Y,\s;\zee[K_+(W)])$ as did the averaged invariant for the original $Z$. 

To state the result for logarithmic transformations, let us fix an identification of $T$ with $S^1\times S^1$ and extend this to an identification $\partial(T\times D^2) \cong S^1\times S^1\times S^1$ such that $pt\times \partial D^2$ corresponds to $pt\times pt \times S^1$. Then the diffeomorphism $\phi$---or, its effect on the first homology of the 3-torus---can be described by a $3\times 3$ matrix; in particular the class $[\phi(\partial D^2)]$ is identified with a vector $(p,q,r)$ of integers. This vector determines the diffeomorphism type of $Z_\phi$, so henceforth we will write $Z(p,q,r)$ for the result of such a logarithmic transformation, when the identification $\partial(nbd(T)) = S^1\times S^1\times S^1$ is understood.

\begin{theorem}\label{logtransfthm} Let $T\subset Z$ be an embedded torus with trivial normal bundle in a smooth 4-manifold $Z$ with boundary $Y$, and let $Z(p,q,r)$ be the result of a logarithmic transformation along $T$ as above. Then for any \spinc structure $\s$ on $Z(p,q,r)$ whose restriction to $T\times D^2$ is trivial, there are choices of representatives for the relative invariants such that 
\begin{equation}\label{rellogtransf}
\Psi_{Z(p,q,r)/T, \s} = p\cdot \Psi_{Z(1,0,0)/T, \s'} + q\cdot \Psi_{Z(0,1,0)/T, \s''} + r\cdot \Psi_{Z(0,0,1)/T, \s'''},
\end{equation}
where in each term $T$ refers to the core torus in $T\times D^2$ considered as lying in the corresponding 4-manifold. Here $\s'$, $\s''$, and $\s'''$ are any \spinc structures on the relevant manifolds whose restriction to $T\times D^2$ is trivial and whose restriction to $W$ agrees with $\s|_W$.
\end{theorem}

\begin{proof}
Recall from \cite{OS4} that the untwisted Floer homology of the 3-torus $T^3$ in the trivial \spinc structure $\s_0$ is given as a module over $\zee[U]\otimes \Lambda^*H_1(T^3;\zee)$ by
\[
HF^-(T^3, \s_0;\zee) \cong \left(\Lambda^2H^1(T^3;\zee)\oplus \Lambda^1H^1(T^3;\zee)\right)\otimes \zee[U],
\]
where $\Lambda^2H^1(T^3)$ lies in absolute degree $-3/2$, $\Lambda^1H^1(T^3)$ is in degree $-5/2$, and as usual $U$ carries degree $-2$. The action of $H_1(T^3)$ is given by the obvious contraction $\Lambda^2H^1(T^3)\to \Lambda^1H^1(T^3)$ in degrees congruent to $-3/2$ modulo 2, and vanishes on $\Lambda^1H^1(T^3)$. We wish to calculate the image $F^-_{T^2\times D^2}(\Theta^-)$, which by the degree shift formula from \cite{OS3} (quoted in equation \eqref{degshift} below) lies in $HF^-_{-3/2}(T^3,\s_0) \cong \Lambda^2H^1(T^3)$.

Let $c\in H_1(T^3)$ be the class $c = [pt\times \partial D^2]$. Then it follows from the $H_1$-equivariance of cobordism-induced maps and the fact that $c$ bounds in $T^2\times D^2$ that $c.F^-_{T^2\times D^2}(\Theta^-) = 0$, i.e., $F^-_{T^2\times D^2}(\Theta^-)$ lies in the one-dimensional subspace of $\Lambda^2H^1(T^3)$ that is the kernel of contraction with $c$. 

It is easy to see that $F^-_{T\times D^2}(\Theta^-)$ must be a primitive class by, for example, embedding $T\times D^2$ as a neighborhood of a regular fiber in an elliptic $K3$ surface and using the composition law for cobordism-induced maps together with the fact that $\O_{K3} = 1$ (c.f. \cite{OSsymp}). Making the natural identification $\Lambda^2H^1(T^3) = H^2(T^3)$, the kernel of contraction by $c$ is generated by the Poincar\'e dual of $c$ itself. Hence we may identify $F^-_{T\times D^2}(\Theta^-)$ as the class Poincar\'e dual to the circle $pt\times \partial D^2$ (up to sign).  

To understand the effect of the gluing diffeomorphism $\phi$ on the $T$-averaged invariant, observe that we can decompose $Z_\phi$ as a composition $T\times D^2\cup C_\phi \cup W$ where the gluings here are identity maps, and $C_\phi$ is the mapping cylinder of $\phi$ thought of as a cobordism $T^3\to T^3$. It follows from the composition law that the homomorphism in untwisted Floer homology induced by $T\times D^2 \cup C_\phi$ is the composition of $F^-_{T\times D^2}$ with $F^-_{C_\phi}$. Write $c^*\in H^2(T^3)$ for the Poincar\'e dual of $c$; then using $H_1$-equivariance it is easy to identify $F^-_{C_\phi}(c^*)$ with $(\phi_*c)^*$. If $\phi$ corresponds to $(p,q,r)$ as in the statement, then the latter class has the expression $p\cdot c_{1,0,0}^* + q\cdot c_{0,1,0}^* + r\cdot c_{0,0,1}^*$, where $c_{1,0,0}^*$, $c_{0,1,0}^*$ and $c_{0,0,1}^*$ are Poincar\'e dual to $S^1\times pt \times pt$, $pt\times S^1\times pt$, and $pt\times pt\times S^1$ respectively (so in particular $c_{0,0,1} = c$). Using this in the definition \eqref{avgdef} gives
\begin{eqnarray*}
F^-_{Z(p,q,r)/T,\s}(\Theta^-) &=& \ul{F}^-_{W}(F^-_{C_\phi}(F^-_{T\times D^2}(\Theta^-))) = \ul{F}^-_{W}((\phi_*c)^*)\\
&=& p\cdot \ul{F}^-_{W}(c_{1,0,0}^*) + q\cdot \ul{F}^-_{W}(c_{0,1,0}^*) + r\cdot \ul{F}^-_{W}(c_{0,0,1}^*)\\
&=& p\cdot \ul{F}^-_{W}(F^-_{C_{1,0,0}}(F^-_{T\times D^2}(\Theta^-))) \\&&\quad {}+ q\cdot \ul{F}^-_{W}(F^-_{C_{0,1,0}}(F^-_{T\times D^2}(\Theta^-))) \\&&\quad {}+ r\cdot \ul{F}^-_{W}(F^-_{C_{0,0,1}}(F^-_{T\times D^2}(\Theta^-))),
\end{eqnarray*}
where $C_{1,0,0}$ etc. are the mapping cylinders of the diffeomorphisms corresponding to the vectors $(1,0,0)$ etc. This is the desired result. 
\end{proof}

\begin{cor}\label{logtransfcor} Let $X$ be a closed 4-manifold, $\eta\in H^2(X,\arr)$ a perturbing class, and $Y\subset X$ a cut for $X$ that is ``allowable'' for $\eta$ in the sense that either $\eta|_Y \neq 0$ or in the decomposition $X = Z_1\cup_Y Z_2$, we have $b^+(Z_1)\geq 1$ and $b^+(Z_2)\geq 1$. Suppose $T\subset Z_1$ is an embedded torus with trivial normal bundle, and assume $\eta = 0$ in a neighborhood of $T$. Let $X(p,q,r)$ denote the result of a generalized logarithmic transformation along $T$. Then there are choices for representatives of the perturbed \OS invariants of the manifolds involved such that
\begin{eqnarray}
\sum \O_{X(p,q,r),Y,\eta,\s} &=& p\cdot \sum \O_{X(1,0,0),Y,\eta',\s'} + q \cdot \sum \O_{X(0,1,0), Y, \eta'', \s''}\nonumber\\ && \hspace*{.5in} + r \cdot \sum \O_{X(0,0,1),Y,\eta''',\s'''}.\label{pertlogtransf}
\end{eqnarray}
Here we consider $Y$ to provide a cut for each of the transformed 4-manifolds above, and choose the data $\s',\s'',\s'''$ and $\eta',\eta'',\eta'''$ to be compatible with the restrictions of $\s$ and $\eta$ away from $T$. In each case the sum is over \spinc structures that agree with $\s$, $\s'$, $\s''$ or $\s'''$ in a neighborhood of $T$ and in the complement of that neighborhood.
\end{cor}

\begin{proof} This follows from the preceding theorem and Lemma \ref{Tavglemma} by pairing both sides of \eqref{rellogtransf} with the relative invariant $\Psi_{Z_2,\eta,\s}$.
\end{proof}

Observe that, if desired, we can obtain a formula analogous to \eqref{rellogtransf} using the untwisted relative invariants $\Psi^{un}$ by applying the map induced by the augmentation homomorphism $\varepsilon: \zee[K_+(W)]\to \zee$ to both sides of \eqref{rellogtransf}. Likewise, by choosing a cut $Y$ for a closed 4-manifold $X$ that is disjoint from the torus on which surgery is being performed, we can obtain a formula for the (individual) \OS invariants of the result $X(p,q,r)$ of a logarithmic transformation on a closed 4-manifold with given \spinc structure $\s_0$:
\[
\sum_\s\Phi_{X(p,q,r),\s} = p\sum_{\s'} \Phi_{X(1,0,0),\s'} + q \sum_{\s''}\Phi_{X(0,1,0),\s''} + r \sum_{\s'''}\Phi_{X(0,0,1),\s'''},
\]
where in each case the sum is over \spinc structures that are trivial over the torus and agree with $\s_0$ away from the surgery region. This formula is in close analogy with a similar one for Seiberg-Witten invariants obtained by Morgan, Mrowka, and Szab\'o \cite{MMS}, but is rather weaker in light of the sign indeterminacy of the invariants $\Phi_X$.
 
In principle, the above follows by equating coefficients on the two sides of \eqref{pertlogtransf}---recall that $\O_{X,Y,\eta}$ generally takes values in a group ring---but there are difficulties with this related to the choice of representatives involved. We skirt that issue by requiring $Y$ to be an ``admissible cut'' in the sense that $H^1(Y)\to H^2(X\setminus T, \partial)$ is trivial (c.f. \cite{OS3}); in this case the group ring involved is simpy $\zee$ and the ambiguity arising from the action of $H^1(Y)$ is eliminated (though the sign issue persists).

\subsection{Fiber sums}
We briefly recall the results we need from \cite{products} regarding fiber sums of 4-manifolds along tori. In contrast with the previous subsection, here we will be interested mainly in tori that are homologically essential: specifically, tori that represent primitive homology classes of infinite order. Let $M$ be a closed 4-manifold and $T\subset M$ an essential torus of self-intersection 0 and infinite order in homology. In this situation the perturbed theory applies, since we can always choose a class $\eta\in H^2(M;\arr)$ pairing nontrivially with the torus that we can use as our perturbation. The formalism of \cite{products} implies that when $b^+(M)\geq 2$, the \OS invariants $\Phi_{M,\s}$ can be calculated as the coefficients of $\O_{M,T^3,\eta}\in \K(M,T^3,\eta)$, where $T^3$ denotes the boundary of a regular neighborhood of the torus $T\subset M$. Here $\O_{M,T^3,\eta}$ is defined by cutting $M$ along $T^3$ and using the pairing in perturbed Floer homology as in \eqref{Odef}, while $\K(M,T^3,\eta)$ is the Novikov completion of $\zee[K(M,T^3)]$ with respect to $\eta$. It is straightforward to see that $K(M,T^3)$ is infinite cyclic and generated by the Poincar\'e dual of the torus $T$, so that $\K(M,T^3,\eta)$ is the ring of Laurent series in a single variable $t$.

If $X = M_1\#_{T_1 = T_2} M_2$ is a fiber sum of two manifolds $M_i$ as in the previous paragraph, then in general $K(X, T^3)$ need not be cyclic: the dual of the class $[T]$ produced by the identification of $T_1$ and $T_2$ generates an infinite cyclic summand, but $K(X,T^3)$ also includes the duals of ``rim tori,'' which are tori in $T^3$ that are nullhomologous in each $M_i$ but essential in $X$. If $\R\subset K(X,T^3)$ is the subgroup generated by the duals of rim tori, there is a natural projection $\rho: K(X,T^3)\to K(X,T^3)/\R$, and we write $\rho$ also for the extension to group rings and Novikov rings. (Here we choose $\eta$ to be compatible with the restrictions of perturbing classes $\eta_i$ on each $M_i$, so in particular it vanishes on rim tori. In this case the extension of $\rho$ to the Novikov ring always exists.) 

The fiber sum formula obtained in \cite{products} reads as follows.

\begin{theorem} Let $X = M_1\#_{T_1 = T_2} M_2$ be the fiber sum of two 4-manifolds $M_1$, $M_2$ along essential tori $T_1$, $T_2$ of square 0, and suppose $\eta\in H^2(X;\arr)$, and $\eta_i\in H^2(M_i;\arr)$ are classes restricting compatibly to the complements of the tori. Furthermore, suppose $\int_T\eta >0$. Then for any \spinc structure $\s$ on $X$,
\[
\rho(\O_{X,T^3, \s,\eta}) = (t^{1/2} - t^{-1/2})^2 \O_{M_1,T^3,\s_1,\eta_1}\cdot \O_{M_2,T^3,\s_2,\eta_2}.
\]
Here $\s_1$, $\s_2$ are \spinc structures on $M_1$, $M_2$ whose restrictions to $M_i\setminus nbd(T_i)$ agree with those of $\s$. The above formula holds up to multiplication by $\pm t^n$, where $t\in K(X, T^3)$ is the class Poincar\'e dual to the fiber sum torus $T_1 = T_2$.
\end{theorem}

\section{Knot Surgery Formulae}\label{knotsurgsec}

We are now in a position to derive analogs for \OS invariants of the well-known formula for the Seiberg-Witten invariants of the result of knot surgery. Here and below, we cannot avoid the sign issues that surfaced previously, and we are obliged to replace $\zee$ by the field $\eff = \zee/2\zee$ of two elements. That is to say, we run the entire preceding package with $\eff[H^1(Y;\zee)]$ replacing $\zee[H^1(Y;\zee)]$, or put another way we form the tensor product over $\zee$ of all our Floer groups with $\eff$. Thus the \OS invariants constructed this way are the mod-2 reductions of the (sign-indeterminate) ones considered above. To emphasize this point, we will use the notation $[\O_{X,Y,\s}]$ for the image of $\O_{X,Y,\s}\in \zee[K(X,Y)]$ in $\eff[K(X,Y)]$. 

 We begin with the case of closed 4-manifolds.

\begin{theorem}\label{closedknotsurg} Let $T\subset X$ be a smoothly embedded torus with trivial normal bundle in a 4-manifold $X$, and suppose there is a class $\eta\in H^2(X;\arr)$ such that $\int_T\eta >0$. On the result $X_K$ of knot surgery along $T$ using $K\subset S^3$, let $\eta_K$ be any 2-dimensional cohomology class whose restriction to the complement of the surgery region agrees with the restriction of $\eta$. Then the perturbed \OS invariants of $X$ and $X_K$ relative to the classes $\eta$, $\eta_K$ and the boundary $T^3$ of the neighborhood of $T$ satisfy
\[
[\O_{X_K, T^3,\s_K,\eta_K}] = [\O_{X,T^3,\s,\eta}\cdot \Delta_K(t)]
\]
up to multiplication by $\pm t^n$. Here $\s$, $\s_K$ are any \spinc structures on $X$, $X_K$ restricting compatibly to the complement of the surgery neighborhood, and $\Delta_K(t)$ is the Alexander polynomial of $K$.
\end{theorem}

\begin{proof} With the log transform and fiber sum formulae in place, this proof is virtually identical to that given by Fintushel and Stern for the case of Seiberg-Witten invariants (for more details on this argument, the reader is referred to \cite{FSknot} or \cite{FSclay}). We reproduce the essentials of the argument here. Recall that the manifold $X_K$ can be described as a fiber sum: 
\[
X_K = X \#_{T = S^1\times m} (S^1\times M_K),
\]
where $M_K$ is the result of 0-surgery along $K\subset S^3$ and $m$ is a meridian for $K$. Thus $S^1\times m$ is an essential torus of square $0$ in $S^1\times M_K$.

 Following \cite{FSknot}, if $L\subset S^3$ is a link with two components we define $X_L$ to be the 4-manifold obtained by fiber sum of $X$ with $S^1\times  s(L)$, where $s(L)$ is the sewn-up exterior of $L$ with appropriate framing (i.e., $s(L)$ is the 3-manifold obtained by identifying the two boundary tori of $S^3\setminus nbd(L)$ using a framing uniquely determined by the condition that $b_1(s(L)) = 2$). The theorem will follow from two properties of $\O_X$: 
\begin{enumerate}
\item $[\O_{X_{K_+}}] = [\O_{X_{K_-}}] + [\O_{X_{K_0}}]$
\item $[\O_{X_{L_+}}] = [\O_{X_{L_-}}] + [\O_{X_{L_0}} \cdot(t^{1/2} - t^{-1/2})^2]$
\end{enumerate}
Here as usual, $K_+$ denotes some one-component knot in a resolution tree for $K$, $K_-$ is the result of changing a positive to a negative crossing in $K_+$, and $K_0$ is the two-component link resulting from resolving the crossing. Likewise $L_+$ and $L_-$ are two-component links differing by a crossing change between strands on different components, and $L_0$ is the knot resulting from resolving the crossing (there is always a resolution tree for $K$ containing only one- and two-component links).

The two relations above imply that if we define a formal series $\Theta_K(t)\in \eff[[t]][t^{-1}]$ by letting $\Theta_K = [\O_{X_K}]/[\O_{X}]$ if $K$ has one component, and $\Theta_L = (t^{1/2} - t^{-1/2})^{-1} [\O_{X_L}]/[\O_{X}]$ if $L$ has two components, then $\Theta_K$ satisfies the same skein relation and normalization as the symmetric Alexander polynomial and the theorem follows (c.f. \cite{FSknot} or \cite{FSclay}).

Property (1) above follows from the log transform formula given in Corollary \ref{logtransfcor}. Indeed, let $c$ denote a smooth unknot linking the relevant crossing of $K_-$, so that $c$ is nullhomologous in the complement of $K\subset S^3$ and such that $K_+$ can be realized as the result of $+1$ Dehn surgery along $c$ (c.f. Figure \ref{crossingfig}(a)).
\begin{figure}[t]
\includegraphics[width=3cm]{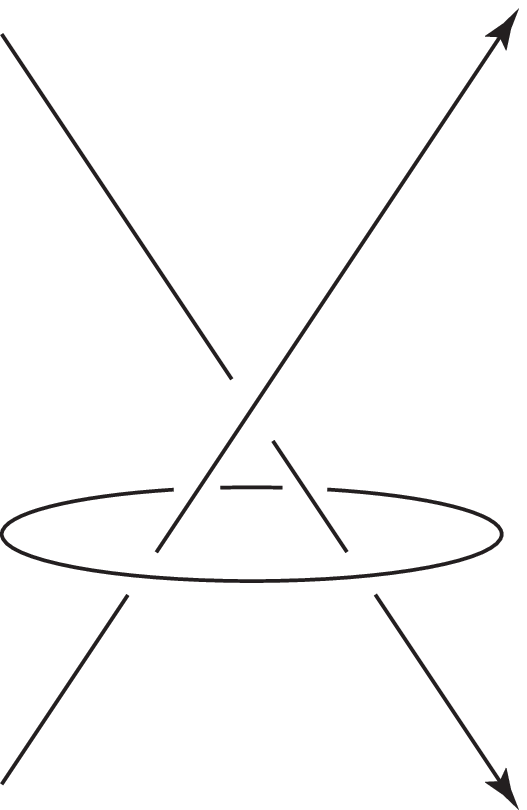}\hfill
\includegraphics[width=3cm]{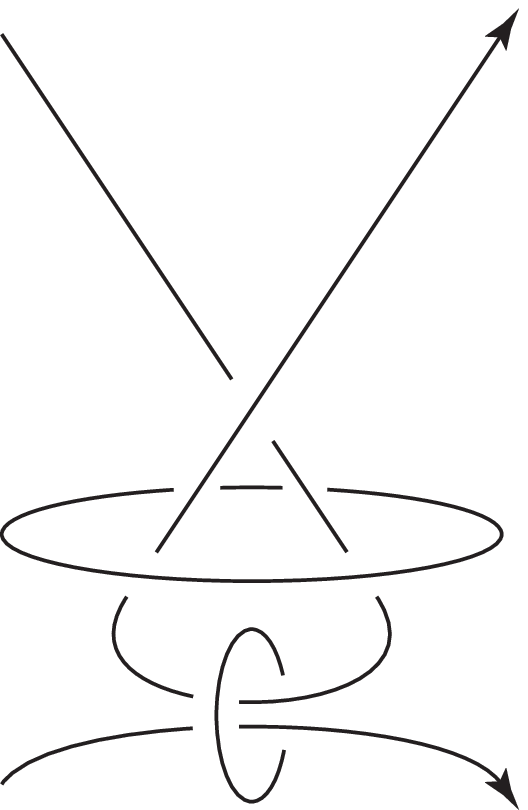}\hfill
\includegraphics[width=3cm]{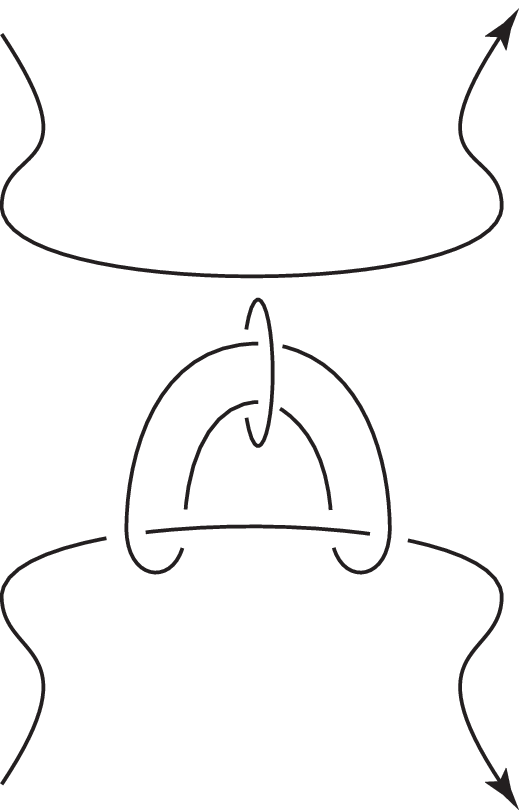}
\put(-285,55){$c$}
\put(-335,105){$K_-$}
\put(-150,55){$c$}
\put(-200,105){$L_0$}
\put(-185,-10){$U$}
\put(-73,105){$0$}
\put(-72,60){$0$}
\put(-45,50){$0$}
\put(-325,-30){(a)}
\put(-190,-30){(b)}
\put(-50,-30){(c)}
\caption{In (a), $+1$-surgery on $c$ yields the knot $K_+$. If the two strands in the picture are on different components of $L_-$, then $s(L_-)$ is given by $0$-surgery on $L_0$ and $U$ in diagram (b). After also performing $0$-surgery on $c$, we can slide $c$ down to encircle the ``waist'' of $U$, then bring the picture to (c) (passing strands of $L_0$ over $c$ as necessary).}\label{crossingfig} 
\end{figure}
Then $X_{K_+}$ is obtained from $X_{K_-}$ by a logarithmic transformation of the form $(p,q,r) = (0,1,1)$, along the torus $S^1\times c$ in the part of $X_{K_-}$ corresponding to $S^1\times M_{K_-}$. With these conventions, $X_{K_-}(0,0,1) = X_{K_-}$, while $X_{K_-}(0,1,0)$ is the fiber sum of $X$ with the result $Z_0$ of a $0$-log transform along $S^1\times c$. That is, $Z_0 = S^1\times M_{K_-}(0)$, where $M_{K_-}(0)$ is the result of 0-surgery along each component of the link $K_-\cup c$. According to Hoste \cite{hoste}, $M_{K_-}(0)$ is the sewn-up exterior of the link $K_0$, so we get property (1), modulo the sums appearing in corollary \ref{logtransfcor}. Since $T$ is nullhomologous in $X_{K_{\pm}}$ the sum corresponding to each of those manifolds contains only one term. On the other hand, there is a torus of square 0 in $X_{K_-}(0)$ intersecting the core torus $T$ in a single point, namely the punctured torus bounded by $c$ in $S^3\setminus K_-$ together with the surgery disk bounding $c$. The existence of this torus implies, by the adjunction inequality, that at most one of the terms in the sum corresponding to $X_{K_-}(0)$ in \eqref{logtransfcor} can be nonzero.

The second property is another application of the log transform formula together with the observation that if $c$ is the circle linking the crossing as above, then $0$-surgery along $c$ in $S^1\times s(L_-)$ gives the fiber sum of $S^1\times M_{L_0}$ with the 4-torus $T^4$. (The author is indebted to Peter Ozsv\'ath for pointing out this fact.) To see this, observe that by Hoste's result again, $s(L)$ is given by $0$-surgery on each component of the link $L_0\cup U$, where $U$ is an unknot linking the two strands of the resolved crossing as shown in Figure \ref{crossingfig}(b).
Performing $0$-surgery along $c$ in addition, we can rearrange the surgery diagram to obtain the manifold $Y$ shown Figure \ref{crossingfig}(c), which is easily seen to be the 3-manifold obtained from $M_{L_0}$ by removing a meridian of $L_0$ and replacing it with the complement of $pt\times pt \times S^1$ in $S^1\times S^1\times S^1 = T^3$. Crossing this picture with $S^1$, we see $S^1\times Y = S^1\times M_{L_0} \#_{f} T^4$ as claimed. Since $\O_{T^4} = 1$, property (2) follows from the fiber sum formula:
\begin{eqnarray*}
[\O_{X_{L_+}}] = [\O_{X_{L_-}(0,1,1)}] &=& [\O_{X_{L_-}(0,1,0)}] + [\O_{X_{L_-}(0,0,1)}] \\
&=& [\O_{X_{L_0}\#_f T^4}] + [\O_{X_{L_-}}] \\
&=& [\O_{X_{L_0}} \cdot (t^{1/2} - t^{-1/2})^2] + [\O_{X_{L_-}}].
\end{eqnarray*}
\end{proof}

\begin{cor}\label{relinvtcor} The 4-manifold $S^1\times M_K$ has perturbed mod-2 \OS invariant given by
\[
[\O_{S^1\times M_K,\s}] = \left[\frac{\Delta_K(t)}{(t^{1/2} - t^{-1/2})^2}\right]
\]
when calculated using the splitting  $S^1\times M_K = (S^1\times (S^3\setminus N(K))) \cup (S^1\times S^1\times D^2)$ and a form $\eta$ pairing positively with the torus $S^1\times m$, where $m$ is a meridian for $K$. Here $\s$ is any \spinc structure pairing trivially with $S^1\times m$; the invariant vanishes for other \spinc structures.

Furthermore, the perturbed relative invariant of $Z = S^1\times (S^3\setminus N(K))$ takes values in $\L(t)$ and is given by
\[
[\Psi_{Z,\s}] = \left[\frac{\Delta_K(t)}{t-1}\right]
\]
up to multiplication by $\pm t^n$.
\end{cor}

\begin{proof} Apply the fiber sum formula to the case of knot surgery on a fiber in an elliptic $K3$ surface. Here $\O_{K3} = 1$, so that
\[
[\Delta_K(t)] = [\O_{K3_K}] = [(t^{1/2} - t^{-1/2})^2 \O_{K3}]\cdot [\O_{S^1\times M_K}] = [(t^{1/2} - t^{-1/2})^2 \O_{S^1\times M_K}].
\]
This gives the first statment; for the second recall from \cite{products} that the relative invariant $\Psi_{Z}$ of the complement of an essential torus of square 0 in $X$ satisfies
\[
\rho(\Psi_Z) = (t-1) \O_X,
\]
where $\rho$ is the projection map $\K(Z)\to \L(t)$ that divides by the duals of rim tori. In the case at hand $Z$ is a homology $T^2\times D^2$, so there are no nontrivial rim tori and the statement follows.
\end{proof}

We can now give a relative version of the knot surgery formula for 4-manifolds with boundary. 

\begin{theorem}\label{knotsurgthm} Suppose $Z$ is a smooth 4-manifold with boundary $Y$, and assume that $T\subset Z$ is an embedded torus with the following properties:
\begin{itemize}
\item $[T]$ is of infinite order in $H_2(Z;\zee)$
\item $[T]$ lies in $Im(H_2(Y;\zee)\to H_2(Z;\zee))$
\end{itemize}
Fix a knot $K$ in $S^3$ and let $Z_K$ denote the result of a knot surgery operation applied to $Z$ along $T$. Let $\s$ be a \spinc structure on $Z$ that is trivial in a neighborhood of $T$, i.e., $\langle c_1(\s), T\rangle = 0$. Then when $b^+(Z)\geq 1$ the (mod-2) relative \OS invariant of $Z$ and $Z_K$ are related by 
\begin{equation}\label{relknotsurgform}
[\Psi_{Z_K,\s_K}] = [\Delta_K(t) \cdot \Psi_{Z,\s}],
\end{equation}
up to multiplication by a unit in $\eff[H^1(Y)]$. Here $\s_K$ is any \spinc structure on $Z_K$ that agrees with $\s$ away from the surgery region and has $\langle c_1(\s_K) , T\rangle = 0$. In the above, $t$ denotes the Poincar\'e dual of any class $[T_0]\in H_2(Y;\zee)$ such that $[T_0]$ maps to $[T]$ under $H_2(Y;\zee)\to H_2(Z;\zee)$. 

When $b^+(Z) = 0$, the equation above holds for the perturbed relative invariants, with respect to any perturbation $\eta$ with $\int_T\eta >0$.
\end{theorem}

Observe that $[T]$ can be thought of as a class in $Z_K$ since there is a natural homology equivalence between $Z$ and $Z_K$. Also note that the multiplication in \eqref{relknotsurgform} is module multiplication between an element $\Delta_K(t) \in \zee[H^1(Y)]$ and the relative invariant $\Psi_{Z,\s}$ which lies in $HF^-(Y,\s; \zee[K(Z)])$. The class $t$ is well-defined up to elements of the image of $H^1(Z)\to H^1(Y)$, but since $K(Z) = H^1(Y)/H^1(Z)$, such elements act trivially on the Floer homology.

\begin{proof} The assumptions of the theorem imply that $T$ represents a nontrivial class in the second homology of both $Z$ and $Y$, with real coefficients. Therefore we can choose a class $\eta\in H^2(Z;\arr)$ pairing positively with $T$, and we fix such a perturbation. Let $N\subset Z$ denote a small tubular neighbhorhood of $T$, and consider the decomposition $Z = N\cup W$ where $W = Z\setminus N$ is a cobordism $T\times S^1\to Y$. There is an analogous decomposition of the surgered 4-manifold, $Z_K = N_K\cup W$, where $N_K = S^1 \times (S^3\setminus nbd(K))$. According to the composition law, we can arrange that
\[
\Psi_{Z_K, \s, \eta} = F^-_{Z_K,\s,\eta}(\Theta^-) = \Pi_*\circ F_{W,\s}^-\circ F^-_{N_K,\s,\eta}(\Theta^-),
\]
where $\Pi_*$ is the homomorphism induced by the coefficient change $\M(N_K)(W)\to \M(Z_K)$. It follows from Corollary \ref{relinvtcor} that $F^-_{N_K,\s,\eta} = \Delta_K(t)\cdot F^-_{N_U,\s,\eta}$ (after tensor product with $\eff$), where $U\subset S^3$ is the unknot (and we replace $\eta$ by a suitable class on $N_U = T^2\times D^2 = N$ without adjusting notation). Since $t$ is dual to a class on $Y$, multiplication by $\Delta_K(t)$ commutes with $F_{W,\s,\eta}^-$ (strictly speaking, $F_{W,\s,\eta}^-\,\Delta_K(t) = \Delta_K(t^{-1})\,F_{W,\s,\eta}^-$, but the Alexander polynomial is symmetric). Likewise $\Pi_*$ is a homomorphism of $\R_{Y,\eta}$-modules, so the above becomes
\[
\Psi_{Z_K, \s, \eta} = \Delta_K(t)\cdot (\Pi_*\circ F_{W,\s}^-\circ F_{N,\s,\eta}^- (\Theta^-)) = \Delta_K(t)\cdot \Psi_{Z,\s,\eta}.
\]
This gives the result for perturbed relative invariants, and the unperturbed case follows since when $b^+(Z)\geq 1$ the two theories agree.
\end{proof}

\section{Floer homology calculations}\label{proofsec}

We return now to the situation of the introduction. Our goal is to use the relative \OS invariants of the complements $Z = X\setminus nbd(\Sigma) $ and $Z_K = X\setminus nbd(\Sigma_K)$ to distinguish the diffeomorphism types of these two 4-manifolds, and therefore in particular the smooth isotopy classes of $\Sigma $ and $\Sigma_K$. From the preceding, we know that the relative invariants for $Z_K$ are related to those for $Z$ by multiplication by the Alexander polynomial of $K$. Thus to carry out our program we must address two points:
\begin{itemize}
\item Show that the relative invariant $\Psi_Z$ is nonzero, at least in certain \spinc structures.
\item Prove that multiplication by $\Delta_K(t)$ has a nontrivial effect on the relative invariant.
\end{itemize}
The second point requires an understanding of the twisted Floer homology group containing $\Psi_Z$, in particular its module structure over $R_Y$. 

In this section we first show that if $X$ is a symplectic 4-manifold with $b^+(X)\geq 2$ and $\Sigma \subset X$ is a symplectic surface whose self-intersection satisfies $|\Sigma.\Sigma|\leq 2g-2$, then the relative invariant of $Z$ in the canonical \spinc structure is a nonzero element of the Floer homology of the boundary. (In fact, the same is shown to be true for arbitrary symplectic manifolds $X$ when $\Sigma.\Sigma < 0$.) Second, we give a complete  calculation of the Floer homology of the boundary of $Z$ in the case that $|\Sigma.\Sigma|>2g-2$, that is, we calculate the twisted Floer homology of circle bundles of ``large'' degree. Two connect these two regimes, we argue in section \ref{blowupsec} below that the property ``the complement of $\Sigma$ has a nonvanishing relative invariant'' is preserved by blowing up, which allows us to pass from a surface of square $2-2g$---whose relative invariant is known to be nonzero---to a surface of square $1-2g$, where the Floer homology is understood.

\subsection{Relative invariants of symplectic surfaces}\label{relinvtsec}

\begin{theorem}\label{nonvanthm} Suppose $\Sigma\subset X$ is a smoothly embedded symplectic surface in a symplectic 4-manifold $X$ with $b^+(X)\geq 2$, and suppose that $n = [\Sigma].[\Sigma]$ satisfies $|n|\leq 2g-2$. Let $Z = X\setminus nbd(\Sigma)$, and let $\kk$ be the restriction of canonical \spinc structure on $X$ to $Z$. Then the relative invariant $\Psi_{Z,\kk}\in HF^-(Y; \zee[K(Z)])$ is nontrivial, where $Y = \partial Z$. In fact, the image $[\Psi_{Z,\kk}]\in HF^-(Y;\eff[K(Z)])$ is also nontrivial.

If $2-2g\leq [\Sigma].[\Sigma]< 0$, then the same is true regardless of the value of $b^+(X)$.
\end{theorem}

Observe that by the adjunction inequality and the nontriviality of \OS invariants of symplectic 4-manifolds with $b^+\geq 2$, we automatically have $n \leq 2g-2$. The hypothesis in the first part of the theorem is therefore an assumption that the square of $\Sigma$ is not ``too negative.''

Note that when $b^+(X) = 1$, one can perform repeated blowup operations if necessary to arrange that $[\Sigma].[\Sigma]$ is negative.

\begin{proof} Recall that the \OS invariant in the canonical \spinc structure $\kk$ on a closed symplectic 4-manifold with $b^+\geq 2$ is equal to $\pm 1$, according to \cite{OSsymp}.

Suppose first that $n>0$ and $b^+(X)\geq 2$. Then $Y = \partial Z$ constitutes an admissible cut for $X$, i.e., the Mayer-Vietoris coboundary $\delta: H^1(Y)\to H^2(X)$ is trivial, and the value of $b^+$ is positive for each component of $X\setminus Y$. In this case (by definition), the \OS invariant $\Phi_{X,\kk}$ is given by
\begin{equation}\label{splitform}
\Phi_{X,\kk} = \langle \tau^{-1}(\Psi_{Z,\kk}), \Psi_{nbd(\Sigma),\kk}\rangle \in \zee,
\end{equation}
a pairing between relative invariants (which, according to \cite{products}, we may take to lie in twisted-coefficient Floer homology groups rather than $\zee$-coefficient Floer homology as in the original definition). Combined with the nonvanishing result of \OnS, we infer that $\Psi_{Z,\kk}\in HF^-(Y, \kk; \zee[K(Z)])$ is a nonzero element of the reduced submodule (and similarly after reducing coefficients modulo 2). 

For negative $n$ this argument does not quite go through: according to \cite{products}, equation \eqref{splitform} holds when $b^+$ is allowed to vanish on one side of $Y$ (as in the case of negative $n$), provided one can pass to perturbed Floer homology: we need a perturbing class $\eta\in H^2(X,\arr)$ restricting nontrivially to $Y$. However, it is easy to see that such a class does not exist if $n\neq 0$. 

We skirt this issue by appealing to gluing results that do not require the perturbed theory. Specifically, let $E$ be a closed symplectic 4-manifold with $b^+(E)\geq 2$, containing an embedded symplectic surface $\Sigma'$ of square $n' = -n >0$. For example, let $S$ be a symplectic smoothing of the union of a section and $g$ disjoint regular fibers in an elliptic $K3$ surface. Then $S$ has genus $g$ and self-intersection $2g-2$, and by blowing up points of $S$ if necessary  (and symplectically splicing $S$ to the exceptional curve) we can reduce the self-intersection to $n'$. (Here we use the hypothesis that $n\geq 2-2g$.) Now form the fiber sum $M = X\#_{\Sigma = \Sigma'} E$. According to Gompf \cite{gompf} and McCarthy-Wolfson \cite{mw}  this manifold has a natural symplectic structure. Furthermore, when $n<0$ we have that $b^+(X\setminus \Sigma) = b^+(X)\geq 1$, and therefore $b^+(M)\geq 2$. We infer that $M$ has nontrivial \OS invariant in the canonical \spinc structure. Since $b^+(X\setminus \Sigma)$ and $b^+(E\setminus \Sigma')$ are both at least 1, we do not need to appeal to the perturbed Floer theory to be able to say that 
\begin{equation}\label{relinvtpair}
\sum_{h\in \delta H^1(Y)} \Phi_{M, \kk + h} e^h = \langle \tau^{-1}(\Psi_{Z,\kk}), \Psi_{Z',\kk}\rangle \in \zee[K(M, Y)],
\end{equation}
where the sum on the left is over the image $K(M,Y)$ of the Mayer-Vietoris boundary $H^1(Y)\to H^2(M)$, and the \spinc structures $\kk$ on the right indicate the restrictions of the canonical \spinc structure on $M$ to $Z$ and $Z' = E\setminus nbd(\Sigma')$ (c.f. \cite{products}).  These restrictions are of course the canonical \spinc structures on the two pieces (by naturality of the symplectic fiber sum), so since the left hand side is nonzero we again infer that $\Psi_{Z,\kk}$ is nonzero in $HF^-(Y,\kk; \zee[K(Z)])$. 
\end{proof} 

\subsection{Twisted Floer homology for circle bundles}\label{calcsec}

Throughout this subsection all Floer homology is to be taken with fully twisted (universal) coefficients, unless specified otherwise.

To calculate the Floer homology of the boundary of the neighborhood of a surface of square $n$, we rely on knot Floer homology techniques as applied in \cite{OSknot,OSsurg,products}. Specifically, recall that for a nullhomologous knot $K\subset Y$ in a \spinc 3-manifold there is a filtered complex $CFK^\infty(Y,K)$, which is just the Heegaard Floer complex for $(Y,\s)$ equipped with an additional filtration induced by the knot. In fact there are two filtrations on $CFK^\infty$, corresponding to the two basepoints in the Heegaard diagram specifying $K$ in $Y$ (strictly speaking, we are implicitly making use of a choice of homology class of Seifert surface for $K$ to give a numerical value to the knot filtration). As typical in the subject, we use the indices $i$ and $j$ to refer to these filtrations (which take the form of a bigrading on the chain complex generating the Floer homology). We will usually refer to $j$ as the ``knot filtration,'' and $i$ as the $U$-filtration for lack of better terminology. Thus for each $i,j$, there is an associated graded complex $CFK^\infty_{i,j}(Y,K)$, whose homology is $\widehat{HFK}(Y,K,j)\otimes U^i$.

According to \cite{OSknot}, the Floer homology groups of a surgery along $K$, with any sufficiently large surgery coefficient $n$, are given by the homology of certain subcomplexes of $CFK^\infty(Y,K)$. ``Sufficiently large'' means in particular that the following descriptions hold for any $|n|>2g-2$, where $g$ is the Seifert genus of $K$. First suppose $n>2g-2$ is positive. Write $\rr_k$ for any \spinc structure on the natural 2-handle cobordism $W: Y_n\to Y$, extending a given structure on $Y$ and having the property that $\langle c_1(\rr_k), [\widehat{F}]\rangle + n = 2k$ modulo $2n$, where $\widehat{F}$ is a Seifert surface for $K$ capped off by the 2-handle in $W$. Let $\s_k$ be the restriction of $\rr_k$ to $Y_n$ (so that $\s_k$ is independent of the choice of $\rr_k$ satisfying the given condition). Then there is an isomorphism of chain complexes $CF^-(Y_n,\s_k) \cong C\{i<0\mbox{ and } j<|k|\}$, where by the latter notation we mean the subcomplex of $CFK^\infty(Y,K)$ for which the filtration values satisfy the indicated constraint. (The construction of $\s_k$ appears to depend on the homology class of the Seifert surface $F$, but the knot filtration also depends on this choice in a precisely cancelling manner. Note also that if the given \spinc structure on $Y$ has torsion Chern class, then $\s_k$ is uniquely determined by the given condition, regardless the choice of Seifert surface.) Correspondingly, we have an isomorphism of $CF^+(Y_n,\s_k)$ with the quotient complex $C\{i\geq 0 \mbox{ or } j\geq |k|\}$ for large positive $n$. In fact, both these isomorphisms are induced by the surgery cobordism itself: the homomorphism $CF^-(Y_n,\s_k)\to CF^-(Y)$ induced by $W$ equipped with a particular choice for $\rr_k$ induces an isomorphism $F_{W,\rr_k}: CF^-(Y_n,\s_k)\to C\{i<0 \mbox{ and } j<|k|\}$ which shifts degree by the factor
\[
\tau_{n,k} =\frac{|n|-(2|k|-|n|)^2}{4|n|}
\]
and similarly for the case of $CF^+$.

When $n$ is large and negative, we have analogous identifications $CF^-(Y_n, \s_k) \cong C\{i<0\mbox{ or } j < -|k|\}$ and $CF^+(Y_n,\s_k)\cong C\{i\geq 0\mbox{ and } j\geq -|k|\}$, and in this case there is a chain isomorphism $F_{W,\rr_k}: C\{i < 0 \mbox{ or } j < -|k|\} \to CF^-(Y_n, \s_k)$ shifting degree by $\tau_{n,k}$ (note that this map is {\it from} the knot complex while in the positive-surgery case the cobordism-induced homomorphism maps {\it to} the knot complex).

Let $Y_n = \partial D_n$ be the boundary of the oriented disk bundle over a surface $\Sigma$ having Euler number $n$. Observe that $Y_n$ can be realized as the result of $n$-framed surgery along a nullhomologous knot $K$ in $Y = \#^{2g}S^1\times S^2$. Specifically, $K$ is the connected sum of $g$ copies of the ``Borromean knot'' $B(0,0)\subset S^1\times S^2\# S^1\times S^2$, which is the third component of the Borromean rings after 0-surgery is performed on the other two components (and $K$ has Seifert genus $g$). The knot Floer homology of $K = \#^g B(0,0)$ in twisted coefficients is given by the exterior algebra
\[
\widehat{HFK}(Y,K, j) \cong \Lambda^{g+j}M,
\]
where $M = R_Y^{2g}$ is a free module of rank $2g$ over the group ring $R_Y = \zee[H^1(\#^{2g}S^1\times S^2)]$ (c.f. \cite{products}, section 9.3). This is graded so that $\Lambda^{g+j}M$ lies in degree $j$.

\begin{remark}\label{stupidremark} There is a natural handle decomposition of the disk bundle $D_n$ having $2g$ 1-handles and a single 2-handle, attached along $\#^g B(0,0)$ with framing $n$. Hence $D_n$ contains a copy of the surgery cobordism $W$ between $Y_n$ and $Y = \#^{2g} S^1\times S^2$, and furthermore (since the Seifert surface for $\#^gB(0,0)$ caps off in $W\subset D_n$ to become the surface $\Sigma$) the \spinc structure $\tilde{\s}_{k}\in Spin^c(D_n)$ characterized by the condition 
\[
\langle c_1(\tilde{\s}_k),[\Sigma]\rangle = 2k - n
\]
 restricts to $W$ as (a choice for) the \spinc structure $\rr_{k}$ used in the characterization of $\s_k\in Spin^c(Y_n)$.
\end{remark}

%Thus we wish to calculate the homology of $C\{i<0\mbox{ and }j< g-1\}$ (in the case of positive $n$) and $C\{i<0 \mbox{ or } j < g-1 \}$ (in the case of negative $n$). Conjugation invariance of Floer homology implies that the latter homology is the same as the homology of $C\{i<0\mbox{ or } j < 1-g\}$, which is more convenient for us to calculate.

Our goal is to calculate the homology of the complexes $C\{i\geq 0 \mbox{ and } j\geq -|k|\}$, etc., mentioned above.
Now, the filtration of $CFK^\infty$ by $i$, what we might call the ``horizontal'' filtration, gives a filtration of each of these complexes, which yields a spectral sequence for the desired homologies. To use this, note first that the $E_1$-term of the spectral sequence coming from the horizontal filtration of the full complex $CFK^\infty$ is given by $\widehat{HF}(Y,R_Y)\otimes \zee[U,U^{-1}] \cong \zee[U,U^{-1}]$, since the fully-twisted Floer homology of $Y = \#^{2g}S^1\times S^2$ is $\widehat{HF}(Y;R_Y) \cong \zee$, supported in degree $-g$. To expand a bit, the $E_1$ term is calculated by taking the homology of $CFK^\infty$ with respect to the ``vertical'' differentials, which in each column $\{i = const\}$ has the form
\begin{equation}\label{freeres}
0\to \Lambda^{2g}M\to \Lambda^{2g-1}M\to \cdots\to \Lambda^1M\to \Lambda^0M\to 0
\end{equation}
(omitting the $U$-powers). Strictly speaking, this complex appears as the $E_1$ term of a ``spectral subsequence'' (coming from the $j$ filtration) calculating the homology of the column, i.e., the $\Lambda^{g+j}M$ are the homologies of the associated graded complexes in a fixed column, namely the knot Floer homology $\widehat{HFK}(Y,K,j)$. For dimensional reasons there can be no differentials beyond $d_1$ in this spectral subsequence, hence the homology of $d_1$ must be the $E_\infty$ term $\widehat{HF}(Y) \cong \zee$.
In fact, in the column $i =0$, the factor $\Lambda^0M = R_Y$ appears in degree $-g$ and since the homology of the column is $\zee$ in degree $-g$ we infer that the above complex provides a free resolution of $\zee$ over $R_Y$---in particular \eqref{freeres} is exact except at $\Lambda^0M$, where the homology is $\zee$. Furthermore, the (full) spectral sequence collapses at this point: the $E_1$ term is already equal to the $E_\infty$ term $HF^\infty(Y;R_Y) = \zee[U,U^{-1}]$, so there can be no further differentials.

As a warmup for our calculation of $HF^+(Y_n,\s_k)$, consider the subcomplex $C\{i\geq 0 \mbox{ and } j\geq -g+1\}$. Every nonzero column except for $i = 0$ contains a copy of the complex \eqref{freeres}, while the $i = 0$ column contains the same complex truncated at $\Lambda^1M$. Hence the $E_1$ term of the spectral sequence has a copy of $\zee$ in grading $-g+2i$ for each $i> 0$, while in the column $i = 0$ we have the homology of 
\[
0\to \Lambda^{2g}M\to \Lambda^{2g-1}M\to \cdots\to \Lambda^1M \to 0,
\]
where $\Lambda^1M$ lies in degree $-g+1$ (see Figure \ref{plusfig}). 
\begin{figure}[t]
\[
\begin{array}{ccccccccccccc}
 & & \Lambda^4 &&&\cdot &\hspace*{1cm} &  & & 0&&&\cdot\\
 & \Lambda^4 & \Lambda^3&&\cdot&&& & 0 & 0&&\cdot&\\
 \Lambda^4 & \Lambda^3 & \Lambda^2&\cdot&&&& 0 & 0 & 0&\cdot&&\\
 \Lambda^3 & \Lambda^2 &\Lambda^1&&&&& 0 & 0 & 0& \\
 \Lambda^2 & \Lambda^1 & \Lambda^0&&&&& 0 & 0 & \zee&\\
 \Lambda^1 & \Lambda^0 & &&&&& Z_1 & \zee && \\
\\
 &\mbox{(a)}&&&&&&&\mbox{(b)} \end{array}
 \]
 \caption{Stages in calculating the homology of $C\{i\geq 0 \mbox{ and } j\geq -g+1\}$, for the case $g = 2$. In each case the leftmost column is $i = 0$. Part (a) shows the complex after taking the $E_0$ homology of each column, where we write $\Lambda^j$ for $\Lambda^j M$, and (b) is the result of taking homology with respect to the vertical differential.}\label{plusfig}
\end{figure}
Since the original complex \eqref{freeres} is exact except at the rightmost point, the complex above has vanishing homology except in the lowest degree, where the homology is a copy of $Z_1:=\Lambda^1M/\im(\Lambda^2M\to \Lambda^1M) \cong \ker(\Lambda^0M\to \zee)$, where the latter is the map arising when viewing \eqref{freeres} as a free resolution of $\zee$. In other words, $Z_1$ is isomorphic to the augmentation ideal $\ker(R_Y\to \zee)$. Now consider the $d_1$ differential in the spectral sequence, the horizontal differential. Clearly the only possible nontrivial differential maps $\zee\to Z_1$ in the row $j = -g+1$. This map is trivial: indeed, there is no nontrivial map of $R_Y$ modules from $\zee$ to the augmentation ideal, since the latter is contained in the free module $R_Y$. Since there can be no differentials further on in the sequence, the spectral sequence collapses at this point. Thus, we see
\[
HF^+_{i+\tau_{n,-g+1}}(Y_n,\s_{-g+1}) = \left\{ \begin{array}{ll}(\T_{-g})_{i} & i\geq -g+2 \\ Z_1 & i = -g+1 \\ 0 & \mbox{otherwise}\end{array}\right.
\]
whenever $n\leq 1-2g$. Here we adopt the standard notational convention that $\T_m$ is the $R_Y\otimes\zee[U]$-module given by
\[
\T_m = \frac{\zee[U,U^{-1}]}{U\cdot\zee[U]},
\]
where elements of $H^1(Y)$ act as the identity, $U$ carries degree $-2$, and $\T_m$ is graded such that the homogeneous summand of lowest degree lies in degree $m$.

We can apply similar reasoning to calculate $HF^-(Y_n, \s_{g-1})$ for {\it positive} $n$. Indeed, this Floer homology is given by the homology of $C\{i<0 \mbox{ and } j< g-1\}$. The homology of the vertical differential in this case is isomorphic to $\zee$ in degree $-g+2i$ for each $i <-1$, while the column $i = -1$ contains the homology of
\[
0\to \Lambda^{2g-1}M\to \Lambda^{2g-2}M\to \cdots\to \Lambda^0M\to 0,
\]
with $\Lambda^0M = R_Y$ lying in degree $-g-2$ (c.f. Figure \ref{minusfig}). 
\begin{figure}[t]
\[
\begin{array}{ccccccccccccc}
&&&& \Lambda^4 & \Lambda^3 &\hspace*{1cm} & &&&& 0 & R_Y \\
&&&\Lambda^4 & \Lambda^3 & \Lambda^2 &&&&& 0 & 0 & 0 \\
&&& \Lambda^3 & \Lambda^2 & \Lambda^1 &&&&& 0 & 0 & 0 \\
&&\cdot& \Lambda^2 & \Lambda^1 & \Lambda^0 &&&&\cdot& 0 & 0 & \zee \\
&\cdot&&\Lambda^1 & \Lambda^0 & &&&\cdot&& 0 & \zee & \\
\cdot&&&\Lambda^0 & & &&\cdot&&& \zee \\\\
&&& & \mbox{(a)} & &&&&& & \mbox{(b)} \\
\end{array}
\]
\caption{Calculation of the homology of $C\{i<0 \mbox{ and } j < g-1\}$ for $g = 2$. Here the rightmost column is $i = -1$. As before (a) is the $E_0$ homology, and (b) is the homology of the vertical differential.}\label{minusfig}
\end{figure}
Exactness of \eqref{freeres} implies that the homology of this sequence is given by $\zee$ in dimension $-g-2$ and $\im(\Lambda^{2g}M\to \Lambda^{2g-1}M)$ in dimension $g-3$. Since $\Lambda^{2g}M = R_Y$ and $\Lambda^{2g}M\to \Lambda^{2g-1}M$ is injective, the homology in dimension $g-3$ is a copy of $R_Y$. Degree considerations show that the spectral sequence collapses here, and we get:
\begin{equation}\label{hf-answer}
HF^-_{i-\tau_{n,g-1}}(Y_n, \s_{g-1}) = \left\{\begin{array}{ll} R_Y & i = g-3 \\ (\zee[U])_{i+g} & i \leq -g-2 \\ 0 & \mbox{otherwise}\end{array}\right. \quad (n\geq 2g-1)
\end{equation}

%
%only one nonzero term, namely $\Lambda^0M = R_Y$ in degree $-g$. The nonzero terms in this $E_1$ stage occur in alternating degrees, so the sequence collapses: we have
%\begin{equation}\label{hf-answer}
%HF^-_{i+\tau_{n,-g+1}}(Y_n, \s_{-g+1}) = \left\{ \begin{array}{ll} R_Y & i = -g \\ (\zee[U])_{i+g} & i < -g \\ 0 & \mbox{otherwise}\end{array}\right.
%\end{equation}

The general calculation we need is given in the following theorem. To state it, we recall that associated to the resolution \eqref{freeres} of $\zee$ as an $R_Y$-module there are ``syzygy modules'' $Z_\ell$, $\ell = 0,\ldots, 2g$, given by
\begin{eqnarray*}
Z_0 & = & \zee \\
Z_1 &=& \ker(\Lambda^0M\to \zee)\\
Z_\ell &=& \ker(\Lambda^{\ell-1}M\to \Lambda^{\ell-2}M) \qquad\quad (\ell\geq 2)\\
&=& \Lambda^\ell M/\im(\Lambda^{\ell+1}M\to \Lambda^\ell M)\\ &=& \im(\Lambda^{\ell}M \to \Lambda^{\ell - 1} M).
\end{eqnarray*}
We have, as previously, $Z_1$ is the augmentation ideal $\ker(\Lambda^0M\to \zee)$, while since the first map in \eqref{freeres} is injective, $Z_{2g} = R_Y$. We will see in the proof of the following that there is an $R_Y$-module homomorphism $\delta: Z_\ell\to Z_{\ell+1}$ for each $\ell$.

\begin{theorem}\label{calcthm} For any $n\leq 1-2g$, the twisted Floer homology of the circle bundle $Y_n$ of degree $n$ over an orientable surface of genus $g\geq 1$ is given as a module over $R_Y \otimes \zee[U]$ by
\[
HF^+_{i+\tau_{n,k}}(Y_n, \s_{k}) = \left\{\begin{array}{ll} (\T_{-g})_{i} & i \geq -|k|+1 \\ Z_{g-|k|}/\delta(Z_{g-|k|-1}) & i = -|k| \\ 0 & \mbox{otherwise}\end{array}\right.
\]
for each $k = -g+1 ,\ldots, g-1$, while if $|k|>g-1$ then $HF^+(Y_n,\s_k)\cong \T_{-g+\tau_{n,k}}$. In particular, for each \spinc structure the reduced Floer homology is supported in a single degree. 
\end{theorem}

Specifically, it will be useful later for us to know that when $n= 1-2g$, the reduced Floer homology of $Y_{1-2g}$ in the \spinc structure $\s_k$ is supported in degree $D(k)$, where
\[
D(k) =-|k| + \tau_{1-2g,k} =  -\frac{k^2}{2g-1} - \frac{g-1}{2}.
\]

\begin{proof} We will proceed using the spectral sequence coming from the filtration by $i$ as in the sample calculations above. Before doing so, however, we make the following observation.

The complex $C\{j\geq 0\}$ is, by results of \cite{OSknot}, quasi-isomorphic to $C\{i\geq 0\}$, and the latter is just the complex $CF^+(Y)$. In particular we know a priori that $H_*(C\{j\geq 0\})\cong \T_{-g}$. Suppose, however, that we wish to calculate this homology using the spectral sequence coming from the $i$ filtration. In a column given by a fixed value of $i$ with $i\geq g$,  vertical differential has homology $\zee$ in dimension $2i-g$. For the columns with $i <g$, the vertical homology is supported in degree $i$ (i.e., the nonzero homology lies along the line $j=0$) and is equal to the syzygy module $Z_{g-i}$. From the structure of this $E_1$ term, it must be the case that the spectral sequence collapses at the $E_2$ stage: that is, the homology of the $E_1$ (horizontal) differential must equal $H_*(C\{j\geq 0\}) = \T_{-g}$. Hence, reading along the row $j =0$, the horizontal differentials give a sequence
\begin{equation}\label{Zseq}
0\lFrom Z_{2g} \lFrom^{\delta_{2g-1}} Z_{2g-1} \lFrom^{\delta_{2g-2}}\cdots\lFrom^{\delta_{2}} Z_2\lFrom^{\delta_1} Z_1 \lFrom^{\delta_0} \zee \lFrom 0
\end{equation}
whose homology is $\zee$ for even-indexed terms and 0 for the others.

Now consider the calculation of $H_*(C\{i\geq 0 \mbox{ and } j\geq k\})$, where by conjugation symmetry we may assume that $k \leq 0$. After taking the vertical homology, we have a ``truncated'' version of the complex just considered: for $i\geq g+k$ the homology is $\zee$ in dimension $2i-g$, while for $0\leq i < g+k$ the homology is $Z_{g+k-i}$ in degree $k+i$, lying along the line $j = k$ (see Figure \ref{genminfig}).
\begin{figure}[bt]
\[
\begin{array}{cccccccc}
\phantom{Z_{g+k-1}}&\phantom{Z_{g+k-1}}&\phantom{Z_{g+k-1}}&\phantom{Z_{g+k-1}}&\phantom{Z_{g+k-1}}&\phantom{Z_{g+k-1}}&0 &\\
&&&&&0 & \vdots &\\
&&&&0&\vdots&0 &\rotatebox{65}{$\ddots$} \\
&&&0&\vdots&0&0 &\\
&&&\vdots&0&0&0& \\
&0&\phantom{\vdots}&0&0&0&0&\\
0&\vdots&&0&0&0&0&\vspace*{-.5ex}\\
\vdots&0&\rotatebox{65}{$\ddots$}&0&0&0&\zee&\\
0&0& \phantom{\vdots}&0&0&\zee\\
Z_{g+k} & Z_{g+k-1} &\phantom{\vdots} & Z_1 & \zee &\phantom{Z_{g+k-1}}&\phantom{Z_{g+k-1}}
\end{array}
\]
\caption{The $E_1$ stage of the spectral sequence calculating the homology of $C\{i\geq 0 \mbox{ or } j\geq k\}$. The leftmost column is $i =0$; the bottom row is $j = k$. The only nontrivial $E_1$ differentials map to the left along the bottom row in a truncated version of \eqref{Zseq}.}\label{genminfig}
\end{figure}
Again, the horizontal differentials must yield the $E_\infty$ term for dimensional reasons, and the only nontrivial horizontal differentials are those between the $Z_\ell$. For $i > 0$, this portion of our complex is identical (up to grading shift) with a part of the sequence \eqref{Zseq} in the previous paragraph---in particular the horizontal differentials have homology $\zee$ or $0$ as above except when $i=0$, and the homology in the latter case is the quotient $Z_{g+k}/\delta_{g+k-1}(Z_{g+k-1})$.
\end{proof}

The twisted Floer homology for circle bundles of large positive degree can be calculated by similar methods, though the argument that the spectral sequence for $HF^+$ collapses at the $E_2$ stage is slightly more complicated in this case. For the sake of completeness we state the result here, but since we will not need it in the sequel we omit the proof.

\begin{theorem} For any $n\geq 2g-1$, the twisted Floer homology of the circle bundle $Y_n$ of degree $n$ over an orientable surface of genus $g\geq 1$ is given as an $R_Y$ module by
\begin{equation}\label{posdeganswer}
HF^+(Y_n,\s_k) \cong K_{g+|k|+1} \oplus \zee[U]/U^{r_k} \oplus \T_{-g-\tau_{n,k}}
\end{equation}
for each $k = -g+1,\ldots,g-1$, while if $|k|>g-1$ then $HF^+(Y_n,\s_k) \cong \T_{-g-\tau_{n,k}}$. 

Here $K_\ell$ is the $R_Y$-module given by
\[
K_\ell = \ker(\delta: Z_{\ell}\to Z_{\ell +1}),
\]
and the factor $K_{g+|k|+1}$ above is supported in degree $|k|-1 - \tau_{n,k}$. Also, $r_k = \lfloor\frac{g-|k|}{2}\rfloor$, and the factor $\zee[U]/U^{r_k}$ is graded such that the nontrivial factor of lowest degree lies in degree $2|k|-g-\tau_{n,k}$.\hfill$\Box$
\end{theorem}

In particular we see that the reduced Floer homology for circle bundles of large positive degree is generally {\it not} supported in a single degree. Since $Y_n$ and $Y_{-n}$ differ only by an orientation change, it is also possible in principle to verify this fact from symmetries of Heegaard Floer homology. Namely, {\OnS} \cite{OS2,OS4} showed that 
\[
HF^+_i(Y, \s;M) \cong HF^{-i-2}_-(-Y,\s;M)
\]
for a torsion \spinc structure $\s$, where the right hand side indicates the {\it cohomology} of the complex $CF^-(Y,\s;M)$. Hence one can use a universal coefficients spectral sequence to provide an alternative calculation of the Floer modules for circle bundles of positive degree starting with those of negative degree, though the homological algebra is somewhat more involved. In brief, the modules $Q_\ell = Z_\ell/\delta(Z_{\ell -1})$ appearing in Theorem \ref{calcthm} have the property that $\mbox{Ext}^q_R(Q_\ell, R)$ is nontrivial for several values of $q$, which gives rise to cohomology in several degrees.

\begin{remark} When $k\equiv g$ modulo 2, the isomorphism \eqref{posdeganswer} respects the action of $U$; in particular $U$ vanishes on $K_{g+|k|+1}$. In the other case, the action of $U$ on $K_{g+|k|+1}$ is not quite clear from our analysis (unless $|k| = g-1$, where $K_{g+|k|+1} = K_{2g} = R_Y$ with trivial $U$-action and the middle summand of \eqref{posdeganswer} drops out).
\end{remark}

\section{Blowing up}\label{blowupsec}

Our intent is to apply Theorem \ref{knotsurgthm} to distinguish the complements of $\Sigma$ and $\Sigma_K$, when $\Delta_K(t)$ is nontrivial. To do so, we need to know that the relative invariant of $Z = X\setminus nbd(\Sigma)$ is nonzero (Theorem \ref{nonvanthm}), and also understand enough about the Floer homology of $Y = \partial Z$ to say that multiplication by $\Delta_K(t)$ has a nontrivial effect on $\Psi_{Z}$, even up to automorphisms of the Floer homology (equation \eqref{hf-answer}). Of course, the two results just mentioned have incompatible assumptions on the square of $\Sigma$. To bridge the gap between them we make the following construction. 

Given a symplectic surface $\Sigma\subset X$ as usual, let $\tX\cong X\#\overline{\cee P}^2$ be the blowup of $X$ at a point of $\Sigma$, and $\tSigma$ the symplectic surface in $\tX$ obtained by symplectically smoothing the total transform of $\Sigma$. Then if $\Sigma$ has square $n$, the blown-up surface $\tSigma$ has square $n-1$. If $\Sigma_K$ is a knotted version of $\Sigma$ (obtained by rim surgery using $K$), and if there is a diffeomorphism $(X,\Sigma_K)\cong (X, \Sigma)$, then certainly $(\tX,\tSigma_K)$ and $(\tX,\tSigma)$ are diffeomorphic---thus it suffices to distinguish the blown-up surfaces. 

Observe that since blowing up is a symplectic operation, there is a canonical \spinc structure $\tkk$ on $\tX$ which satisfies $c_1(\tkk) = c_1(\kk) - E$ where $\kk$ is the canonical \spinc structure on $X$ and $E$ is the Poincar\'e dual of the exceptional sphere ${\mathbb C}P^1\subset \cpbar$.

Let $U\subset \tX$ be a regular neighborhood of the union of $\Sigma$ and the exceptional curve $E$ of the blowup (i.e., $U$ is a neighborhood of the total transform of $\Sigma$). Then the complement of $U$ in $\tX$ is naturally diffeomorphic to $Z = X\setminus nbd(\Sigma)$; we will identify these two manifolds. Let $V\subset\subset U$ be a regular neighborhood of the smoothing $\tSigma$, and let $W = \overline{U}\setminus V$. If $Y_m$ denotes the circle bundle over a genus $g$ surface having Euler number $m$, then we can (with appropriate orientation choice) view $W$ as a cobordism from $Y_{-n}$ to $Y_{-n+1}$, and the complement $\tZ = \tX\setminus nbd(\tSigma)$ is given by $\tZ = Z \cup_{Y_{-n}} W$. For the next result we assume that the genus of $\Sigma$ is at least 2, and for technical reasons that the restriction map $H^1(Z;\zee)\to H^1(\partial Z;\zee)$ is trivial (e.g., the complement of $\Sigma$ is simply connected).

\begin{theorem}\label{blowupthm} Assume, in the situation above, that $n = 2-2g$. Let $\kk$ be the restriction of the canonical \spinc structure on $X$ to $Z$, and let $\tkk_+$ (resp. $\tkk_-$) be the restriction of the canonical \spinc structure $\tkk$ on $\tX$ to $\tZ$ (resp. the restriction of the \spinc structure $\tkk + E$). Then the relative invariant of $\tZ$ is nonzero in at least one of the \spinc structures $\tkk_\pm$. 
\end{theorem}

It is worth pointing out that the restriction $\tkk_\pm|_Z$ is just the canonical \spinc structure $\kk$ on $Z$.

The result above might be seen as an analogue for manifolds with boundary of the fact that blowing up preserves \OS invariants of closed manifolds---though of course it is not the case that $\tZ$ is a blowup of $Z$.

\begin{proof} We study the handle structure of $W$ (for much of this proof, we view $W$ as mapping the opposite direction from previously, i.e., think of $W$ as a cobordism from $Y_{n-1}$ to $Y_n$). There is a standard handle picture for the neighborhood of a genus $g$ surface of square $n$ containing $2g$ 1-handles and a single 2-handle; the corresponding surgery diagram is the connected sum of $g$ copies of $B(0,0)$, with framing $n$. The blowup $U$ of this neighborhood is obtained by adding a 2-handle attached to a disjoint unknot with framing $-1$; sliding the $n$-framed Borromean knot over this 2-handle gives a diagram in which the surface $\tSigma$ is visible. Hence in words, $W$ consists of a single 2-handle, attached along the meridian of the $n-1$-framed Borromean knot with framing $-1$, which is to say it is the ``standard'' surgery cobordism from $Y_{n-1}$ to $Y_{n}$ (thinking of the latter as the results of surgery on $\#^gB(0,0)$). Figure \ref{Wpicture} depicts this handle description for $W$ in the case $g = 2$.
\begin{figure}
\includegraphics[width=10cm]{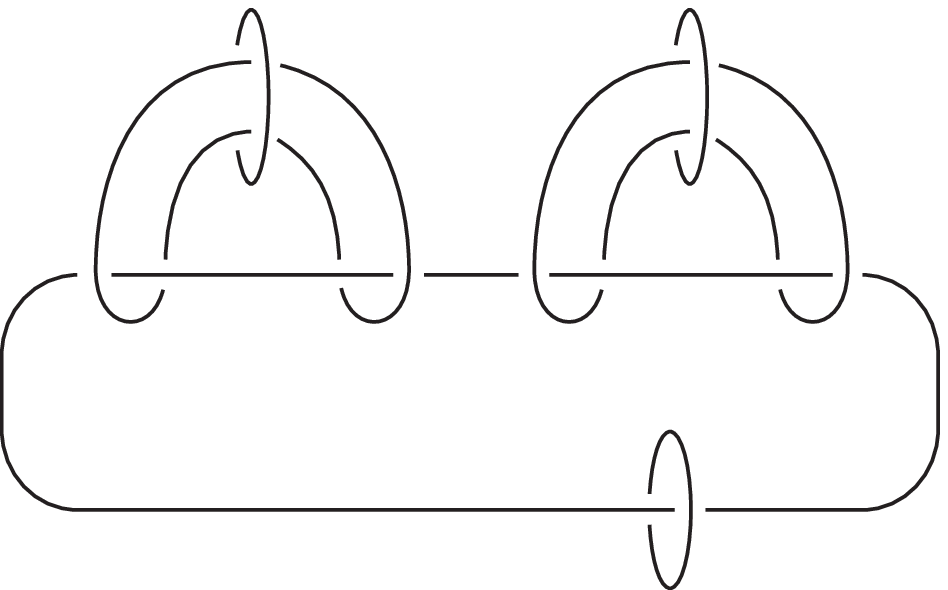}
\put(-68,175){$\langle 0\rangle$}
\put(-200,175){$\langle 0\rangle$}
\put(-25,125){$\langle 0\rangle$}
\put(-275,125){$\langle 0\rangle$}
\put(-73,0){$-1$}
\put(5,60){$\langle n-1\rangle$}
\caption{Handlebody diagram for $W: Y_{n-1}\to Y_n$.}\label{Wpicture}
\end{figure}
The brackets in Figure \ref{Wpicture} follow the notation of Gompf and Stipsicz \cite{GS}, and indicate that the diagram describes a 4-dimensional cobordism built by adding a single 2-handle to the 3-manifold given by the bracketed surgery diagram.

Now, the relative invariant (in the canonical \spinc structure) of the complement of a surface of square $n$ in a symplectic 4-manifold is an element of $HF^-_{d_n^-}(Y_{-n}, \kk)$, where $d_n^-$ is a rational number that is calculated explicitly below. (Here the circle bundle $Y_{-n}$ appears rather than $Y_n$ since we take its orientation to be induced by the complement of the surface.) The pairing \eqref{pairingdef} in the case of a torsion \spinc structure, e.g. the restriction of $\kk$ to $Y_n$, is nontrivial only on the factors $HF^+_i \otimes_{R_Y} HF^-_{-i-2}$ (c.f. \cite{OS3}): thus the ``dual'' summand of the Floer homology with respect to that pairing is $HF^+_{d_n^+}(Y_n,\kk)$ where $d_n^+ = -d_n^--2$. For the rest of the proof we specialize to the case at hand, namely $n = 2-2g$.

%Recall that the relative invariant of a 4-manifold with boundary is the element of the Floer homology of the boundary obained by mapping a top-degree generator $\Theta^-\in HF^-(S^3)$ into the twisted Floer homology of the boundary via the map in $HF^-$ induced by the cobordism obtained by removing a ball from the manifold in question; when $b^+>0$ the relative invariant lies in the reduced submodule. Recall also that if $W$ is a cobordism with \spinc structure $\s$ that restricts as a torsion structure on the boundary, the induced homomorphism shifts degree by the quantity $d(\s) = \frac{1}{4}(c_1^2(\s) - 3\sigma(W) - 2e(W))$, where $\sigma(W)$ is the signature of the intersection form on $W$ and $e(W)$ is the Euler characteristic. If we take $W$ to be obtained from a symplectic 4-manifold by removing a small 4-ball along with the (disjoint) neighborhood of a surface of square $n$, we see that the relative invariant for the complement of the surface, in the canonical \spinc structure, lies in $HF^-_{d_n^-}(Y_{-n}, \kk)$ where
%\[
%d_n^- = \ts\frac{1}{4n}(-4n -(2g-2-n)^2 - (1+4g)n).
%\]
%(This makes use of the fact that for the canonical \spinc structure $\kk$ on a closed symplectic manifold $X$, $c_1^2(\kk) = 3\sigma(X) + 2e(X)$).

We will see below that the two \spinc structures $\tkk_\pm$ (restricted to $W$) have the following properties:
\begin{itemize}
\item The restrictions of $\tkk_\pm$ to $Y_{1-2g}$ are the two \spinc structures $\s_{\pm(g-1)}$ in the notation of the previous subsection (in particular, they are conjugate \spinc structures).
\item The sum of the two maps in fully-twisted Floer homology induced by $(W,\tkk_+)$ and $(W,\tkk_-)$ gives a {\it surjection}
\[
F^+_{W, \tkk_\pm}: HF^+_{d_{1-2g}^+}(Y_{1-2g},\tkk_+)\oplus HF^+_{d_{1-2g}^+}(Y_{1-2g},\tkk_-) \to HF^+_{red, d_{2-2g}^+}(Y_{2-2g},\kk).
\]
\end{itemize}
In the above, $\tkk_\pm$ and $\kk$ refer to the restrictions of $\tkk_\pm$ to $Y_{1-2g}$ and $Y_{2-2g}$ (in the latter case, this is the same as the restriction of the canonical \spinc structure on $X$ to the boundary of the neighborhood of $\Sigma$). Additionally, we are implicitly composing the cobordism-induced homomorphism $F^+_{W,\tkk_\pm}$ with the projection $HF^+\to HF^+_{red}$ on the range side.

Granted these two points, we claim that the map
\[
F^-_{W,\tkk_\pm}: HF^-_{d_{2-2g}^-}(Y_{2g-2}, \kk)\to HF^-_{d_{1-2g}^-}(Y_{2g-1},\tkk_+)\oplus HF^-_{d_{1-2g}^-}(Y_{2g-1},\tkk_-)
\]
is nonzero on the relative invariant $\Psi_{Z,\kk}$. (Observe that since $H^1(Z)\to H^1(Y_{2g-2})$ is trivial, the relative invariant does lie in the fully-twisted Floer homology.)  Indeed, suppose that $F^-_{W,\tkk_\pm}(\Psi_{Z,\kk}) = 0$. Then duality for maps in twisted-coefficient Floer homology \cite{products} implies that for any $\xi \in HF^+(Y_{1-2g}, \tkk_{\pm})$ we have 
\[
\langle F^+_{W,\tkk_{\pm}}(\xi), \Psi_{Z,\kk}\rangle = \langle \xi, F^-_{W,\tkk_\pm}(\Psi_{Z,\kk})\rangle = 0.
\]
Since $F^+_{W,\tkk_\pm}$ is onto the reduced Floer homology in the relevant degree, this means that $\Psi_{Z,\kk}$ pairs trivially with every member of $HF^+_{red}(Y_{2-2g},\kk)$. But according to the fiber sum argument in the proof of Theorem \ref{nonvanthm} (c.f. \eqref{relinvtpair}), there must be an element of $HF^+_{red}(Y_{2-2g},\kk)$ pairing nontrivially with $\Psi_{Z,\kk}$. This contradiction implies that $F^-_{W,\tkk_\pm}(\Psi_{Z,\kk}) \neq 0$, i.e., at least one of $\alpha^+ = F^-_{W,\tkk_+}(\Psi_{Z,\kk})$ or $\alpha^- = F^-_{W,\tkk_-}(\Psi_{Z,\kk})$ is a nonzero element of $HF^-(Y_{2g-1}, \tkk_+)$ (resp. $HF^-(Y_{2g-1}, \tkk_-)$).

By the composition law, the element $\alpha^+$ is exactly the relative invariant for $\tZ$ in the \spinc structure $\tkk_+$, and similarly for $\alpha^-$. This proves the theorem, modulo the two points above.
\end{proof}

The two points deferred in the above proof follow from straightforward but somewhat laborious calculations in cohomology, combined with the structure of the Floer homology of circle bundles derived in the last subsection. We begin by calculating the degree $d_n^-$ of the relative invariant of the complement of a surface of square $n$ in a symplectic manifold. 

First, recall that if $X$ is a closed symplectic manifold with canonical \spinc stucture $\kk$, then $c_1^2(\kk) = 3\sigma(X) + 2e(X)$ where $\sigma(X)$ is the signature of $X$ and $e(X)$ the Euler characteristic. On the other hand, if $W$ is a cobordism from $Y_1$ to $Y_2$ and $\s\in Spin^c(W)$ is a \spinc structure such that the restrictions of $c_1(\s)$ to each of $Y_1$ and $Y_2$ is a torsion class, then by \cite{OS3}, the induced homomorphism in Floer homology shifts degree by
\begin{equation}\label{degshift}
d(\s) = \ts\frac{1}{4}(c_1^2(\s) - 3\sigma(W) - 2e(W)).
\end{equation}

In particular, if we remove two disjoint 4-balls from a closed symplectic 4-manifold $X$, then the resulting cobordism $S^3\to S^3$ has $d(\kk) = 1$. 

Now suppose $\Sigma\subset X$ is a symplectic surface: then the adjunction formula says that $\langle c_1(\kk), \Sigma\rangle = 2g-2-n$, where $n = \Sigma.\Sigma$ is the self-intersection of $\Sigma$. Let $N$ be a tubular neighborhood of $\Sigma$ in $X$ and $\kk_N$ the restriction of the canonical \spinc structure to $N$, so that $c_1(\kk_N)$ is equal to $2g-2-n$ times the generator of $H^2(N;\zee)$ (the generator being specified as the Kronecker dual of $[\Sigma]$). As a cohomology class whose restriction to $\partial N$ is torsion, the Chern class of $\kk_N$ has a well-defined square in rational cohomology, which is
\[
c_1^2(\kk_N) = \frac{(2g-2-n)^2}{n}.
\]
Assume that $n<0$, so that $\sigma(N) = -1$. We see $e(N\setminus B^4) = 1-2g$, from which it follows that the degree shift induced by $N\setminus B^4$ with \spinc structure $\kk_N$ is
\[
d(\kk_N) = \frac{1}{4n}((2g-2-n)^2 + (1+4g)n).
\]
Additivity of the degree shifts shows that if $Z = X\setminus (N\sqcup B^4)$ is the complement of $\Sigma$ (with a small ball removed) then $d(\kk_Z) = 1-d(\kk_N)$. Hence, since the generator $\Theta^-\in HF^-(S^3)$ lies in degree $-2$, the relative invariant $\Psi_{Z,\kk_Z} = F^-_{Z,\kk_Z}(\Theta^-)$ lies in degree
\[
d_n^- = -2+d(\kk_Z) = \frac{1}{4n}(-4n- (2g-2-n)^2 - (1+4g)n)
\]
if $n<0$. Taking $n = 2-2g$ in particular this shows:

\begin{lemma} The relative invariant of the complement of a symplectic surface of genus $g$ and self-intersection $2-2g$, equipped with the canonical \spinc structure, lies in degree $d_{2-2g}^-$ of the Floer homology of the circle bundle $Y_{2g-2}$ over $\Sigma$, $HF^-_{d_{2-2g}^-}(Y_{2g-2},\kk)$, where
\[
d_{2-2g}^- = -\frac{13}{4} + g.
\]
The corresponding dual group is $HF^+_{d_{2-2g}^+}(Y_{2-2g},\kk)$, where 
\[
d_{2-2g}^+ = -d_{2-2g}^- - 2 = \frac{5}{4} - g.
\]\hfill$\Box$
\end{lemma}

To prove the second bulleted claim in the proof of Theorem \ref{blowupthm}, we first analyze the topology of $W$ a bit further (the first bulleted claim will come out along the way). As before, let us write $D_n$ for the disk bundle over $\Sigma$ having degree $n$, and for a fixed $n$ define $M = D_{n}\# \cpbar$, the blowup of $D_{n}$ at a point of $\Sigma$ (identified with the zero-section in $D_{n}$). Tubing $\Sigma$ to the exceptional curve gives the blown-up surface $\tSigma$ of self-intersection $n-1$, and the complement of the neighborhood $D_{n-1}$ of $\tSigma$ is the cobordism $W$ under consideration. Thus 
\[
M = D_{n}\#\cpbar= D_{n-1} \cup_{Y_{n-1}} W.
\]

By way of notation, write $s_{n}$ and $e$ for the generators of $H_2(M;\zee)$ given by the homology classes of $\Sigma$ and the exceptional sphere $\cee P^1\subset \cpbar$. Let $s_{n-1}\in H_2(D_{n-1};\zee)$ be the generator equal to the class of $\tSigma$. 

It is a simple matter using the sequence in homology for the pair $(W, Y_{n-1})$ to see that $H_2(W;\zee) \cong A \oplus \zee\langle a\rangle$, where $A \cong \zee^{2g}$ is the image of $H_2(Y_{n-1};\zee)$ in $H_2(W)$, and $\zee\langle a \rangle$ indicates an infinite cyclic summand generated by a class $a$ whose image in $H_2(W,Y_{n-1})$ is $n-1$ times a generator. Now, a portion of the homology sequence for $(M,W)$ reads
\begin{diagram}
&\rTo& H_2(W)& \rTo &H_2(M)&\rTo^B &H_2(M,W)&\rTo\\
&\rTo& A\oplus \zee\langle a \rangle & \rTo &\zee\oplus\zee &\rTo& \zee &\rTo
\end{diagram}
where we observe that by excision, $H_2(M,W)\cong H_2(D_{n-1},Y_{n-1})$ is generated by the class $d$ of a normal disk to $\Sigma$ in $D_{n-1}$. Hence the homomorphism $B$ above is given by $x\mapsto (x. s_{n-1})\, d$, where $x\in H_2(M)$ and $x.s_{n-1}$ denotes the intersection pairing. Since $s_{n-1} = s_{n} - e$, we infer that $B$ acts on the generators of $H_2(M)$ by
\[
B(s_{n})= n\, d \qquad\mbox{and}\qquad B(e)= d.
\]
Under the homomorphism $H_2(W)\to H_2(M)$, it is easy to see that the subgroup $A$ maps to $0$, and thus we can arrange that $a$ maps to the generator $s_{n} - ne$ of $\ker(B)$. From this we can infer in particular that the nontrivial portion of the intersection form on $W$ is given by the self-intersection 
\begin{equation}\label{square}
a^2 =(s_{n} - n e)^2 = -n(n-1).
\end{equation}

Turning now to the Mayer-Vietoris sequence in homology for the decomposition $M = D_{n-1} \cup W$, we have
\begin{diagram}
&\rTo &H_2(Y_{n-1}) & \rTo & H_2(D_{n-1})\oplus H_2(W) & \rTo^C & H_2(M) & \rTo\\
& \rTo & \zee^{2g} & \rTo & \zee\langle s_{n-1}\rangle \oplus (A\oplus \zee\langle a\rangle) & \rTo & \zee\langle s_{n},e\rangle & \rTo
\end{diagram}
where on the factor $\zee\langle s_{n-1},a\rangle$, $C$ is represented in the given bases by the matrix
\[
C = \left[\begin{array}{cc} 1 & 1 \\ -1 & -n\end{array}\right].
\]

Passing to cohomology, let us write $s_{n}^*, e^*$ for the basis of $H^2(M;\zee)$ Kronecker dual to $s_{n},e\in H_2(M;\zee)$, and similar for the other relevant groups. Since $s_{n}$ and $e$ are represented by symplectic surfaces, the adjunction formula quickly shows that the canonical class $\tK = c_1(\tkk)$ is 
\[
\tK = (2g-2-n) s^*_{n} - e^*.
\]
Applying the transpose $C^*$ of the homomorphism $C$, we have
\[
C^*(\tK) = (2g-1-n)s_{n-1}^* \oplus (2g-2)a^* \in H^2(D_{n-1})\oplus H^2(W),
\]
and in particular the restriction of the canonical class to $W$ is $\tK|_W = (2g-2)a^*$. 

\begin{lemma}\label{restrlemma} Let $\{\s_k\}$, $k \in \zee/(n-1)\zee$, resp. $\{\tt_\ell\}$, $\ell \in \zee/n\zee$, denote the collection of \spinc structures on $Y_{n-1}$ (resp. $Y_{n}$) defined by the condition that $\s_k$ is the restriction of any \spinc structure $\rr_k$ on the disk bundle $D_{n-1}$ over $\Sigma$ having $\langle c_1(\rr_k),[\Sigma]\rangle = 2k - (n-1)$ modulo $2(n-1)$ (resp., the same conditions with $n-1$ replaced by $n$). Let $W: Y_{n-1}\to Y_{n}$ be the cobordism considered above, equipped with its symplectic structure as above and write $\tkk$ for the canonical \spinc structure on $W$. 

Then $\tkk$ defines a \spinc cobordism between $\s_{g-1}$ and $\tt_{g-1}$, and has $c_1(\tkk) = (2g-2)a^*$ where $a^*\in H^2(W;\zee)$ is the generator described above. More generally, if $\rr_\ell\in Spin^c(W)$ has $c_1(\rr_\ell) = 2\ell a^*$, then $\rr_\ell$ interpolates between $\s_\ell$ and $\tt_\ell$ (where we reduce $\ell$ modulo $n-1$ and $n$, respectively).
\end{lemma}

\begin{proof} We have just calculated the Chern class of the canonical \spinc structure; the fact that $\tkk$ connects $\s_{g-1}$ and $\tt_{g-1}$ is a quick consequence of the adjunction formula. To prove the general case, observe first that any \spinc structure $\rr$ on $W$ satisfying the hypotheses can be written $\rr = \tkk + ma^*$ for some $m\in \zee$.

Now, the \spinc structures $\s_k\in Spin^c(Y_{n-1})$ have the property that $\s_{k+1} = \s_k + s_{n-1}^*$, where $s_{n-1}^*$ is (the restriction to the boundary of) the generator of $H^2(D_{n-1})$ considered above (and similarly for the $\tt_\ell$). Hence the lemma follows from the observation that 
\begin{equation}\label{restrictions}
a^*|_{Y_{n-1}} = s_{n-1}^*|_{Y_{n-1}} \qquad\mbox{and}\qquad a^*|_{Y_n} = s_n^*|_{Y_n}.
\end{equation}
These two facts in turn are consequences of the analysis above. Indeed, thinking of $W$ as a subspace of $M = D_{n}\#\cpbar = D_{n-1} \cup_{Y_{n-1}} W$, then according to the preceding
\[
C^*: H^2(M)\to H^2(D_{n-1})\oplus H^2(W)
\]
has $C^*(s^*_{n}) = s_{n-1}^* \oplus a^*$. In particular, the restriction of $s_{n}^*$ to $W$ is equal to $a^*$, so the second equation in \eqref{restrictions} holds. Exactness of the Mayer-Vietoris sequence then shows that if $\rho = \rho_{D_{n-1}} - \rho_{W}: H^2(D_{n-1}) \oplus H^2(W)\to H^2(Y_{n-1})$ is the restriction, then $0 = \rho(s_{n-1}^*\oplus a^*) = \rho_{D_{n-1}}(s_{n-1}^*) - \rho_W(a^*)$, giving the first part of \eqref{restrictions}. 
\end{proof}

We now turn to the maps in Floer homology induced by $W$, in the case $n = 2-2g$. As observed previously, $W$ is the standard surgery cobordism from $Y_{1-2g}$ to $Y_{2-2g}$ (c.f. Figure \ref{Wpicture}), and as such it induces the homomorphism at the top of the following surgery exact triangle (with fully-twisted coefficients understood):
\[
\begin{diagram}
\bigoplus HF^+(Y_{1-2g}) && \rTo^{F_W^+} & & \bigoplus HF^+(Y_{2-2g})\\
&\luTo_f & & \ldTo & \\
&& HF^+(Y) &&
\end{diagram}
\]
Here the direct sums are over all \spinc structures on the circle bundles $Y_{1-2g}$, $Y_{2-2g}$  extending the torsion \spinc structure on $Y= \#^{2g}S^1\times S^2$, and the map $F_W^+$ is the sum of the homomorphisms induced by $W$ with all choices of \spinc structure. We claim first that $f: HF^+(Y)\to \bigoplus HF^+(Y_{1-2g})$ is injective.

To see this, it suffices to show that a single component of $f$ is injective. We consider the component mapping into $HF^+(Y_{1-2g}, \s_{-g})$, where as usual, $\s_k\in Spin^c(Y_{1-2g})$ is characterized as the restriction of any \spinc structure $\rr$ on the cobordism $W_0: Y\to Y_{1-2g}$ satisfying $\langle c_1(\rr), [S]\rangle - (1-2g) = 2k$ modulo $2(1-2g)$, where $S\subset W_0$ is a capped-off Seifert surface for the knot $\#^g B(0,0)$. (By remark \ref{stupidremark}, our uses of the symbol $\s_k$ are consistent.) This component of $f$ is the sum of homomorphisms induced by all such $\rr\in Spin^c(W_0)$. We have seen that 
\[
HF^+(Y_{1-2g}, \s_k) \cong H_*(C\{i\geq 0 \mbox{ and } j\geq k\}),
\]
and furthermore it is proved in \cite{OSknot} that when $k \neq 0$ the top-degree component of the map $HF^+(Y)\to HF^+(Y_{1-2g}, \s_k)$ induced by $W_0$ is given by the map in homology coming from the natural projection
\[
C\{i\geq 0\} \to C\{i\geq 0 \mbox{ and } j\geq k\}.
\]
That is to say, this projection gives one of the homomorphisms whose sum equals the component of $f$ under consideration, and the remaing homomorphisms have strictly lower degree. Taking $k = -g$, and recalling that the knot Floer homology is supported between the lines $j = i \pm g$, this projection is obviously an isomorphism. It follows that the component of $f$ mapping into $HF^+(Y_{1-2g}, \s_{-g})$ is also an isomorphism, and injectivity of $f$ follows.

We infer that the map $F_W^+$ in the surgery triangle is surjective. This fact alone, however, is insufficient for our purposes since $F_W^+$ is a sum over the maps induced by all \spinc structures on $W$: indeed, we are interested particularly in the \spinc structures $\rr\in Spin^c(W)$ whose restriction to $Y_{2-2g}$ is $\tt_{g-1}$ (i.e., the same as that of the canonical \spinc structure). According to Lemma \ref{restrlemma}, if $\rr\in Spin^c(W)$ interpolates between $\s\in Spin^c(Y_{n-1})$ and $\tt\in Spin^c(Y_n)$, then so does $\rr + (2g-1)(2g-2)a^*$. Observe also that if $\rr$ connects $\s_k$ to $\tt$, for $\s_k$ as in Lemma \ref{restrlemma}, then $\rr + (2g-2)a^*$ connects $\s_{k-1}$ to $\tt$. 

We define a 2-parameter family of \spinc structures on $W$ as follows. First let $\rr_{g-1,0} = \kk$ be the canonical \spinc structure, and set $\rr_{g-1,m} = \rr_{g-1,0} + m(2g-1)(2g-2)a^*$. Thus $\{\rr_{g-1,m}\}_{m \in {\mathbb Z}}$ is the family of \spinc structures connecing $\s_{g-1}$ to $\tt_{g-1}$. Now let $\rr_{\ell, 0} = \rr_{g-1,0} + (2g-2)(g-1-\ell)a^*$, so that according to the above, $\rr_{\ell,0}$ interpolates between $\s_\ell$ and $\tt_{g-1}$. Finally, define
\[
\rr_{\ell,m} = \rr_{g-1,0} + [(2g-2)(g-1-\ell) + (2g-1)(2g-2)m]a^*,
\]
so that for a fixed $\ell$, the collection $\{\rr_{\ell,m}\}_{m\in {\mathbb Z}}$ is the family of \spinc structures on $W$ connecting $\s_\ell$ to $\tt_{g-1}$. 

It is worth pausing here to observe that if $\rr\in Spin^c(W)$ has $c_1(\rr) = 2m a^*$, then since $\sigma(W) = -1$ and $e(W) = 1$, the map in Floer homology induced by $\rr$ shifts degree by
\[
d(\rr) = \frac{1}{4}\left(-\frac{4m^2}{(2g-1)(2g-2)} +1\right)
\]
(c.f. \eqref{square}). In particular, if $\{\rr_j\}$ is a family of \spinc structures on $W$ with $c_1(\rr_j) = 2m_j a^*$, then the degree shift of the corresponding homomorphisms is maximized by that $\rr_j$ for which the corresponding $m_j$ is closest to 0. 

With this in mind, we observe that since $c_1(\rr_{g-1,0}) = c_1(\kk) = (2g-2)a^*$ we have
\begin{equation}\label{c1eqn}
c_1(\rr_{\ell,m})= (2g-2)((2g-1)(2m+1)-2\ell)a^*.
\end{equation}
Fixing $\ell\in \{-g+1,\ldots,g-1\}$, we infer that the maximal degree shift induced by a \spinc structure interpolating between $\s_\ell$ and $\tt_{g-1}$ is the one given by $\rr_{\ell,m_\ell}$, where $m_\ell$ is the closest integer to the solution of $(2g-1)(2m+1) - 2\ell = 0$, i.e.,
\[
m_\ell = \left[ \frac{\ell}{2g-1} - \frac{1}{2}\right] = \left\{\begin{array}{ll} 0 & \mbox{if $\ell >0$} \\ 0 \mbox{ or } -1 & \mbox{if $\ell = 0$} \\ -1 & \mbox{if $\ell <0$}\end{array}\right.
\]
Since we will be interested mainly in the \spinc structures with maximal degree shift, let us define
\[
\rr_\ell = \left\{\begin{array}{ll} \rr_{\ell, 0} & \mbox{if $\ell \geq 0$} \\ \rr_{\ell, -1} & \mbox{if $\ell\leq 0$}.\end{array}\right.
\]
Of course there is an ambiguity here when $\ell = 0$, but the \spinc structures $\rr_{0,0}$ and $\rr_{0,-1}$ are conjugate (as follows from \eqref{c1eqn}, for example) and therefore have the same degree shift, and we will see that the distinction between these structures is unimportant.

\begin{prop} Define
\[
\tkk_\pm = \rr_{\pm(g-1)} \in Spin^c(W).
\]
Then the homomorphism in fully-twisted Floer homology
\[
F^+_{W,\tkk_+} + F^+_{W,\tkk_-}: HF^+(Y_{1-2g}, \s_{g-1})\oplus HF^+(Y_{1-2g},\s_{-g+1})\to HF^+(Y_{2-2g}, \tt_{g-1})
\]
induces a surjection from the lowest nontrivial degree in the domain group to the portion of the  reduced module $HF^+_{red}(Y_{2-2g},\tt_{g-1})$ lying in degree $d_{2-2g}^+$. %Furthermore, if $\rr\in Spin^c(W)$ has $\rr\neq \tkk_\pm$, then $F^+_{W,\rr}$ vanishes on the lowest-degree summand of $HF^+(Y_{1-2g},\rr|_{Y_{1-2g}})$.
\end{prop}

In the last part of the statement, we are implicitly composing the cobordism-induced maps with the natural projection $HF^+\to HF^+_{red}$. The notation of the proposition is chosen to agree with that in Theorem \ref{blowupthm}; in the notation of Lemma \ref{restrlemma}, $\tkk_+$ is the canonical \spinc structure $\tkk$. Indeed, since $e^*|_W = (2g-2)a^*$ (according to the discussion before Lemma \ref{restrlemma}), we have that
\[
c_1(\rr_{g-1}) - c_1(\rr_{-g+1}) = 2(e^*|_W) = -2(E|_W),
\]
where $E$ is the Poincar\'e dual of the exceptional curve. (Thus the first bulleted claim in the proof of Theorem \ref{blowupthm} is proved.)

\begin{proof} We already know that the sum of the homomorphisms induced by $W$ with all possible \spinc structures is surjective. Furthermore, from Theorem \ref{calcthm} we have that the reduced Floer homology of $Y_{1-2g}$ in a fixed \spinc structure is supported in a single degree, the lowest. Since we are interested only in the image of the homomorphisms after projection to the reduced module, these two facts imply that to prove the proposition it suffices to prove: Unless $\rr = \rr_{\pm(g-1)}$, the homomorphism $F^+_{W,\rr}$ maps the lowest-degree summand of $HF^+(Y_{1-2g})$ into $HF^+_d(Y_{2-2g})$, where $d < d^+_{2-2g}$.

This fact follows from a simple calculation of degree shifts. Indeed, we may restrict attention to the maximally-shifting structures $\rr_\ell$, $\ell = -g+1,\ldots, g-1$, on $W$ (defined above), and by conjugation-invariance we may also assume $\ell\geq 0$. According to the degree shift formula, the corresponding homomorphisms shift degree by the quantity
\begin{eqnarray*}
d(\rr_\ell) &=& \frac{1}{4}(c_1^2(\rr_\ell) +1)\\
&=& -\frac{2g-2}{2g-1}\ell^2 + (2g-2)\ell - \frac{1}{4}((2g-2)(2g-1) - 1)
\end{eqnarray*}
(c.f. \eqref{square} and \eqref{c1eqn}).

From Theorem \ref{calcthm}, the Floer homology of $Y_{1-2g}$ in the \spinc structure $\s_\ell = \rr_\ell|_{Y_{1-2g}}$ has lowest nontrivial grading equal to 
\[
D(\ell) = -\frac{\ell^2}{2g-1} - \frac{g-1}{2}.
\]

Hence the image of the lowest-degree part of $HF^+(Y_{1-2g},\s_\ell)$ under $F^+_{W,\rr_\ell}$ lies in degree
\[
D(\ell) + d(\rr_\ell) = -\ell^2 + (2g-2)\ell - \frac{1}{4}(2g-1)^2.
\]
This quadratic function of $\ell$ is maximized for $\ell = g-1$, and is strictly increasing for $\ell< g-1$. Furthermore, if $\ell = g-1$ a quick check shows $D(g-1) + d(\rr_{g-1}) = \frac{5}{4} - g = d^+_{2-2g}$. Thus these degree considerations show that among all \spinc structures on $W$, the only ones that can map onto the reduced Floer homology in degree $d^+_{2-2g}$ are $\rr_{\pm(g-1)}$.
\end{proof}

This completes the proof of Theorem \ref{blowupthm}.

\begin{proof}[Proof of Theorem \ref{mainthm}] Let $K_n$, $n = 1,2,\ldots$ be a sequence of knots in $S^3$ whose Alexander polynomials are all distinct after reducing the coefficients modulo 2 (e.g., take $K_n$ to be any fibered knot of Seifert genus $n$), and let $\Sigma_n = \Sigma_{K_n}$ be the surface obtained from $\Sigma_0$ by rim surgery using $K_n$ and some fixed circle on $\Sigma_0$. We claim that no two pairs $(X, \Sigma_n)$ are diffeomorphic. 

To see this, recall from the beginning of this section that it suffices to distinguish the complements of the knotted surfaces after arbitrarily many blowups. In particular, we may assume by blowing up if necessary that the self-intersection of $\Sigma_0$ is $2-2g$.

If $g = 1$, we have a torus of square zero and much of the preceding work does not apply. Of course, it is not necessary either: in this case we know that the relative invariant of the complement is nontrivial, and the reader may verify that the top-degree portion of the Floer homology $HF^-(T^3, \s_0; R_{T^3})$ is a free module of rank 1. Hence multiplication by the Alexander polynomial has a nontrivial effect, and indeed we can distinguish infinitely many knotted surfaces this way. Since this case does not extend the results of Fintushel and Stern, we leave the details to the reader and assume from now on that $g>1$.

Now, since $\Sigma_0$ is symplectic, Theorem \ref{nonvanthm} shows that $Z_0 = X\setminus nbd(\Sigma_0)$ has a nonvanishing relative invariant in the canonical \spinc structure (since the complement of $\Sigma_0$ is simply-connected, the relative invariant lives in the fully-twisted Floer homology). Then according to Theorem \ref{blowupthm}, the complement $\tZ_0$ of the surface $\tSigma_0$ obtained by blowing up one more time has a nonvanishing invariant in at least one of the ``blown-up canonical \spinc structures'' $\tkk_\pm$, which restrict to $\partial \tZ_0 = Y_{2g-1}$ as $\s_{\pm(g-1)}$. Furthermore, this invariant lies in the topmost degree of the Floer homology of the boundary, either by the degree shift formula or by the fact that the relative invariant lies in the reduced submodule, which is contained in the topmost degree (c.f. \eqref{hf-answer}).

Consider the following invariant $\cO_{\tZ_0}$ of $\tZ_0$:
\[
{\cO}_{\tZ_0} = \sum_{\rr\in Spin^c(\tZ_0)} \Psi_{\tZ_0,\rr} \cdot e^{c_1(\rr)} \quad \in HF^-(Y_{2g-1},\s_{\pm (g-1)})[H^2(\tZ_0;\zee)],
\]
where the sum is over all \spinc structures $\rr$ on $\tZ_0$ such that $\rr|_{\partial \tZ_0} = \s_{\pm(g-1)}$ and we may, if desired, project to the top-degree part of the Floer homology (note that since $\tZ_0$ is simply-connected, \spinc structures are determined by their Chern classes). Observe that the sum above is finite since $b^+(\tZ_0)\geq 1$ (c.f. Theorem 3.3 and Lemma 8.2 of \cite{OS3}).

Since the complement of $\tSigma_n$ is given by knot surgery on $\tZ_0$, Theorem \ref{knotsurgthm} shows that after passing to coefficients in $\eff[H^1(Y_{2g-1})]$,
\[
[\cO_{\tZ_n}] = [\Delta_{K_n}(t)\cdot \cO_{\tZ_0}],
\]
where $t$ is dual to the rim torus.

Now, $\cO_{\tZ_n}$ is an invariant of $\tZ_n$ up to automorphisms of $H^2(\tZ_n;\zee)$ induced by diffeomorphisms, and module automorphisms of $HF^-(Y_{2g-1},\s_{\pm(g-1)})$ (which either respect or reverse the direct sum decomposition). According to \eqref{hf-answer}, the top-degree part of the latter Floer homology is a free module of rank 1,
\[
HF^-_{top}(Y_{2g-1},\s_{g-1})= HF^-_{top}(Y_{2g-1},\s_{-g+1})= R_{Y_{2g-1}}.
\]
It follows immediately that if $[\Delta_{K_n}(t)] \neq 1$ then $[\cO_{\tZ_n}]$ and $[\cO_{\tZ_0}]$ distinguish the smooth types of $\tZ_n$ and $\tZ_0$, and more generally that $\tZ_n$ and $\tZ_m$ are smoothly distinct if $\Delta_{K_n}(t)$ and $\Delta_{K_m}(t)$ are distinct (with coefficients taken modulo 2). 
\end{proof}

\end{document}